\documentclass[12pt]{article}
\usepackage{e-jc}
\usepackage{hyperref}

\usepackage{amssymb,amsthm,amsmath,amsfonts,epsfig,latexsym,color}
 
\newtheorem{theorem}{Theorem}[section]
\newtheorem{definition}[theorem]{Definition}
\newtheorem{proposition}[theorem]{Proposition}
\newtheorem{lemma}[theorem]{Lemma}
\newtheorem{corollary}[theorem]{Corollary}

\newtheorem{problem}[theorem]{Problem}

\theoremstyle{definition}

\newtheorem{example}[theorem]{Example}

\def\N{\ensuremath{\mathbb{N}}}

\def\A{\ensuremath{\mathbf{A}}}
\def\B{\ensuremath{\mathbf{B}}}
 
\def\E{\ensuremath{\mathbf{E}}}

\def\Gbar{\ensuremath{\overline{G}}}
\def\T{\ensuremath{TTT}}
\def\Tug{\ensuremath{\mathrm{Tug}}}
\def\Ult{\ensuremath{\mathrm{Ult}}}

\def\<{\ensuremath{\langle}}
\def\>{\ensuremath{\rangle}}

\date{\dateline{Feb 28, 2008}{May 26, 2010}{Jun 10, 2010} }

\begin{document}

\title{Discrete bidding games}

\author{Mike Develin\\
\small American Institute of Mathematics \\ [-0.8ex]
\small 360 Portage Ave., Palo Alto, CA 94306 \\
\small \texttt{develin@post.harvard.edu} \\
\and
Sam Payne\thanks{Supported by the Clay Mathematics Institute.}\\
\small Stanford University, Dept. of Mathematics \\ [-0.8ex]
\small 450 Serra Mall, Stanford, CA 94305 \\
\small \texttt{spayne@stanford.edu}
}

\begin{abstract}
We study variations on combinatorial games in which, instead of alternating moves, the players bid with discrete bidding chips for the right to determine who moves next.  We consider both symmetric and partisan games, and explore differences between discrete bidding games and \emph{Richman games}, which allow real-valued bidding.  Unlike Richman games, discrete bidding game variations of many familiar games, such as chess, Connect Four, and even Tic-Tac-Toe, are suitable for recreational play.  We also present an analysis of Tic-Tac-Toe for both discrete and real-valued bidding.
\end{abstract}

\maketitle

\tableofcontents

\section{Introduction}

Imagine playing your favorite two-player game, such as Tic-Tac-Toe, Connect Four, or chess, but instead of alternating moves you bid against your opponent for the right to decide who moves next.  For instance, you might play a game of bidding chess in which you and your opponent each start with one hundred bidding chips.  If you bid twelve for the first move, and your opponent bids ten, then you give twelve chips to your opponent and make the first move.  Now you have eighty-eight chips and your opponent has one hundred and twelve, and you bid for the second move...

\vspace{5pt}

Similar bidding games were studied by David Richman in the late 1980s.  In Richman's theory, as developed after Richman's death in \cite{LLPU, LLPSU}, a player may bid any nonnegative real number up to his current supply of bidding resources.  The player making the highest bid gives the amount of that bid to the other player and makes the next move in the game.  If the bids are tied, then a coin flip determines which player wins the bid.  The goal is always to make a winning move in the game; bidding resources have no value after the game ends.  The original Richman theory requires that the games be \emph{symmetric}, with all legal moves available to both players, to avoid the possibility of \emph{zugzwang}, positions where neither player wants to make the next move. The theory of these real-valued bidding games, now known as Richman games, is simple and elegant with surprising connections to random turn games.  The recreational games, such as chess, that motivated the work presented here, are partisan rather than symmetric, and it is sometimes desirable to force your opponent to move rather than to make a move yourself.  However, the basic results and arguments of Richman game theory go through unchanged for partisan games, in spite of the remarks in \cite[p.~260]{LLPSU}, provided that one allows the winner of the bid either to move or to force his opponent to move, at his pleasure.  For the remainder of the paper, we refer to these possibly partisan real-valued bidding games as \emph{Richman games}.

Say Alice and Bob are playing a Richman game, whose underlying combinatorial game is $G$.  Then there is a critical threshold $R(G)$, sometimes called the \emph{Richman value} of the game, such that Alice has a winning strategy if her proportion of the total bidding resources is greater than $R(G)$, and she does not have a winning strategy if her proportion of the bidding resources is less than $R(G)$.  If her proportion of the bidding resources is exactly $R(G)$, then the outcome may depend on coin flips.

The critical thresholds $R(G)$ have two key properties, as follows.  We say that $G$ is finite if there are only finitely many possible positions in the game, and we write $\Gbar$ be the game that is just like $G$ except that Alice and Bob have exchanged roles.  Let $P(G)$ be 
the probability that Alice can win $G$ at random-turn play, where the player who makes each move is determined by the toss of a fair coin, assuming optimal play.

\begin{enumerate}
\item If $G$ is finite, then $R(G)$ is rational and equal to $1-R(\Gbar)$.

\item For any $G$, $R(G)$ is equal to $1-P(G)$.
\end{enumerate}

\noindent The surprising part of (1) is that, if $G$ is modeled on a finite graph, which may contain many directed cycles, there is never a range of distributions of bidding resources in which both Alice and Bob can prolong the game indefinitely and force a draw.  On the other hand, for infinite games $R(G)$ can be any real number between zero and one \cite[p. 256]{LLPSU}, and $1-R(\Gbar)$ can be any real number between zero and $R(G)$; in particular, the Richman threshold may be irrational and there may be an arbitrarily large range in which both players can force a draw.  The connection with random turn games given by (2) is especially intriguing given recent work connecting random turn selection games with conformal geometry and ideas from statistical mechanics \cite{PSSW07}.

The discrete bidding variations on games that we study here arose through recreational play, as a way to add spice and interest to old-fashioned two player games such as chess and Tic-Tac-Toe.  The real valued bidding and symmetric play in Richman's original theory are mathematically convenient, but poorly suited for recreational play, since most recreational games are partisan and no one wants to keep track of bids like  $e^{\sqrt \pi} + \log 17$.  Bidding with a relatively small number of discrete chips, on the other hand, is easy to implement recreationally and leads to interesting subtleties.  For instance, ties happen frequently with discrete bidding with small numbers of chips, so the tie-breaking method is especially important.  To avoid the element of chance in flipping coins, we introduce a deterministic tie-breaking method, which we call the \emph{tie-breaking advantage}.  If the bids are tied, the player who has the tie-breaking advantage has the choice either to declare himself the winner of the bid and give the tie-breaking advantage to the other player, or declare the other player the winner of the bid and keep the tie-breaking advantage.  See Section~\ref{tie-breaking section} for more details.  As mentioned earlier, partisan games still behave well under bidding variations provided that the winner of the bid has the option of forcing the other player to move in zugzwang positions.

Other natural versions of bidding in combinatorial game play are possible, and some have been studied fruitfully.  The most prominent example is Berlekamp's ``economist's view of combinatorial games" \cite{Berlekamp96},  which is closely related to Conway's theory of thermography \cite{Conway76} and has led to important advances in understanding Go endgames.  

\bigskip

Since this paper was written, Bidding Chess has achieved some popularity among fans of Chess variations \cite{Beasley08, Beasley08b}.  Also, bidding versions of Tic-Tac-Toe and Hex have been developed for recreational play online, by Jay Bhat and Deyan Simeonov.  Readers are warmly invited to play against the computer at 
\begin{center}
\href{http://bttt.bidding-games.com/online/}{\texttt{http:/\!/bttt.bidding-games.com/online/}} 
\end{center}
\noindent and
\begin{center}
\href{http://hex.bidding-games.com/online/}{\texttt{http:/\!/hex.bidding-games.com/online/}},
\end{center}
\noindent and to challenge friends through Facebook at 
\begin{center}
\href{http://apps.facebook.com/biddingttt}{\small \texttt{http:/\!/apps.facebook.com/biddingttt}} \ and \
\href{http://apps.facebook.com/biddinghex}{\small \texttt{http:/\!/apps.facebook.com/biddinghex}}.
\end{center}
\noindent The artificial intelligence for the computer opponent in Bidding Hex is based on the analysis of Random-Turn Hex in \cite{PSSW07} and connections between random turn games and Richman games, and is presented in detail in \cite{biddinghexai}.  Although we cannot prove that the algorithm converges to an optimal or near-optimal strategy, it has been overwhelmingly effective against human opponents.

\subsection{A game of bidding Tic-Tac-Toe}

We conclude the introduction with two examples of sample bidding games.  First, here is a game of Tic-Tac-Toe in which each player starts with four bidding chips, and Alice starts with tie-breaking advantage.  In Tic-Tac-Toe there is no zugzwang, so the players are simply bidding for the right to move.
\vspace{ 3pt}

\noindent \emph{First move.} Both players bid one for the first move, and Alice chooses to use the tie-breaking advantage, placing a red {\color{red} A} in the center of the board. \vspace{3pt}

\noindent  \emph{Second move.}  Now Alice has three chips, and Bob has five chips plus the tie-breaking advantage.  Once again, both players bid one.  Bob uses the tie-breaking advantage and places a blue {\color{blue} B} in the upper-left corner.

\vspace{ 3pt}

\noindent
\emph{Third move.} Now Alice has four chips plus the tie-breaking advantage, while Bob has four chips.  Alice bids two, and Bob also bids two.  This time Alice decides to keep the tie-breaking advantage, and lets Bob make the move.  Bob places a blue {\color{blue} B} in the upper-right corner, threatening to make three in a row across the top.  The position after three moves is shown in the following figure.

\begin{center}
\includegraphics{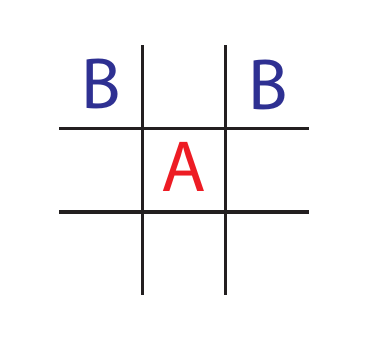}
\end{center}
\vspace{-5 pt}

\noindent
\emph{Fourth move.}  Now Alice has six chips plus the tie-breaking advantage, and Bob has two chips.  Bob is one move away from winning, so he bets everything, and Alice must give him two chips, plus the tie-breaking advantage, to put a red {\color{red} A} in the top center and stop him.

\vspace{ 3pt}

\noindent
\emph{Conclusion.}  Now Alice has four chips, and Bob has four chips plus the tie-breaking advantage.  Alice is one move away from victory and bets everything, so Bob must also bet everything, plus use the tie-breaking advantage, to move bottom center and stop her.  Now Alice has all eight chips, plus the tie-breaking advantage, and she coolly hands over the tie-breaking advantage, followed by a single chip, as she moves center left and then center right to win the game.
\vspace{5pt}

Normal Tic-Tac-Toe tends to end in a draw, and Alice and Bob started with equal numbers of chips, so it seems that the game should have ended in a draw if both players played well.  But Alice won decisively.  \emph{What did Bob do wrong?}

\subsection{A game of bidding chess.}

Here we present an actual game of bidding chess, played in the common room of the mathematics department at UC Berkeley, in October 2006.  Names have been changed for reasons the reader may imagine.

Alice and Bob each start with one hundred bidding chips.  Alice offers Bob the tie-breaking advantage, but he declines.  Alice shrugs, accepts the tie-breaking advantage, and starts pondering the value of the first move.  Alice is playing black, and Bob is playing white.

\vspace{5pt}

\noindent
\emph{First move.}  After a few minutes of thought on both sides, Alice bids twelve and Bob bids thirteen for the first move.  So Bob wins the bid, and moves his knight to c6.  Now Alice has one hundred and thirteen chips and the tiebreaking advantage, and Bob has eighty seven chips.
\vspace{5pt}

\noindent
\emph{Second move.}  Alice figures that the second move must be worth no more than the first, since it would be foolish to bid more than thirteen and end up in a symmetric position with fewer chips than Bob.  She decides to bid eleven, which seems safe, and Bob bids eleven as well.  Alice chooses to use the tie-breaking advantage and moves her pawn to e3.  Bob, who played chess competitively as a teenager, is puzzled by this conservative opening move.
\vspace{5pt}

\noindent
\emph{Third move.}  Now Alice has one hundred and two chips, and Bob has ninety eight and the tie-breaking advantage.  Sensing the conservative tone, Bob decides to bid nine.  He is somewhat surprised when Alice bids fifteen.  Alice moves her bishop to c4.  The resulting position is shown below.  

\vspace{5pt}
\begin{center}
\scalebox{.4}{\includegraphics{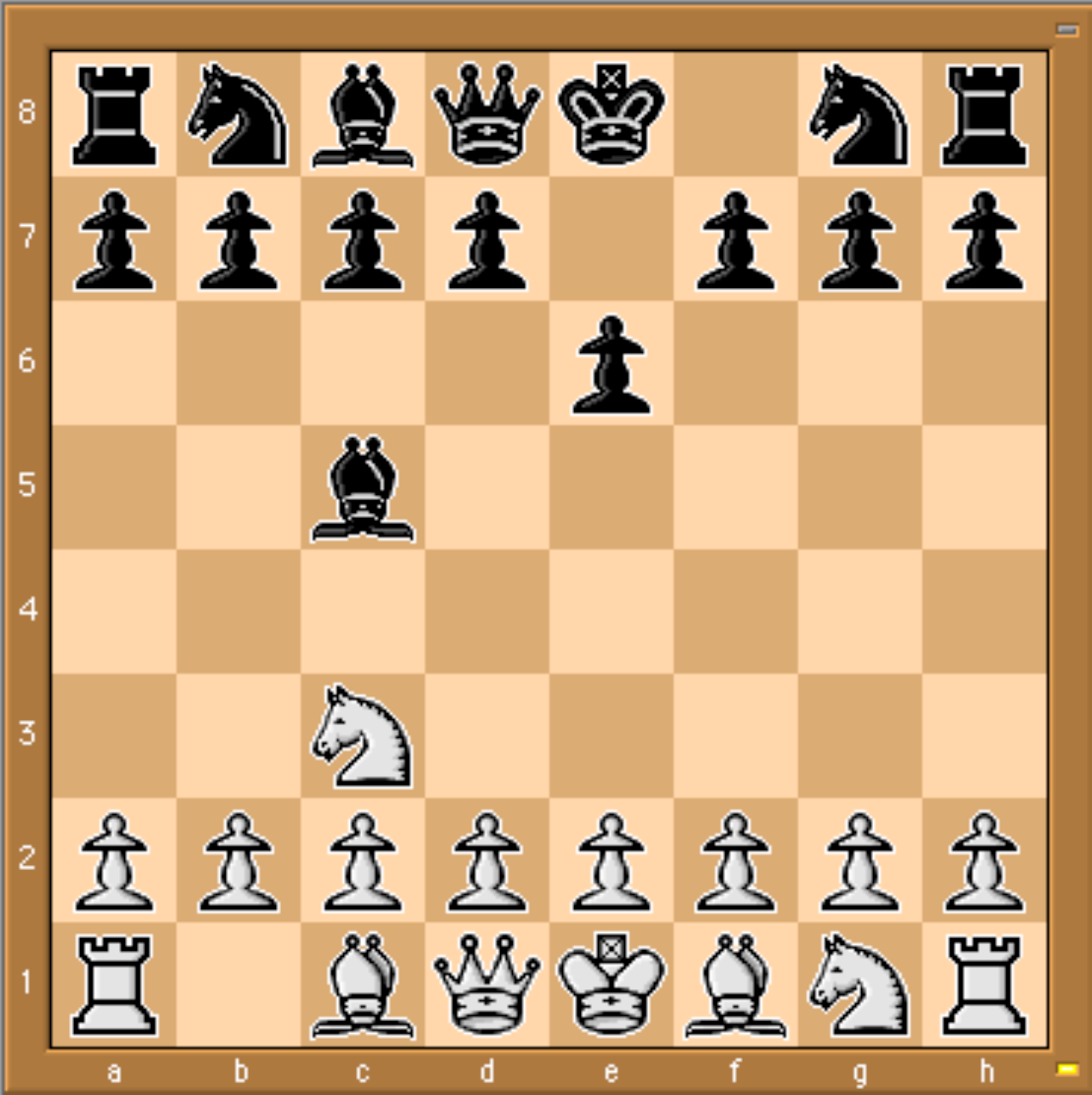}}
\end{center}
\vspace{5pt} 

\noindent
\emph{Fourth move.}  Now Alice has eighty seven chips, while Bob has one hundred and thirteen and the tie-breaking advantage.  Since Alice won the last move for fifteen and started an attack that he would like to counter, Bob bids fifteen for the next move.  Alice bids twenty two, and takes the pawn at f7.  Bob realizes with some dismay that he must win the next move to prevent Alice from taking his king, so he bids sixty five, to match Alice's total chip count, and uses the tie-breaking advantage to win the bid and take Alice's bishop with his king.  The resulting position after five moves is shown here.
\vspace{5pt}

\vspace{5pt}
\begin{center}
\scalebox{.4}{\includegraphics{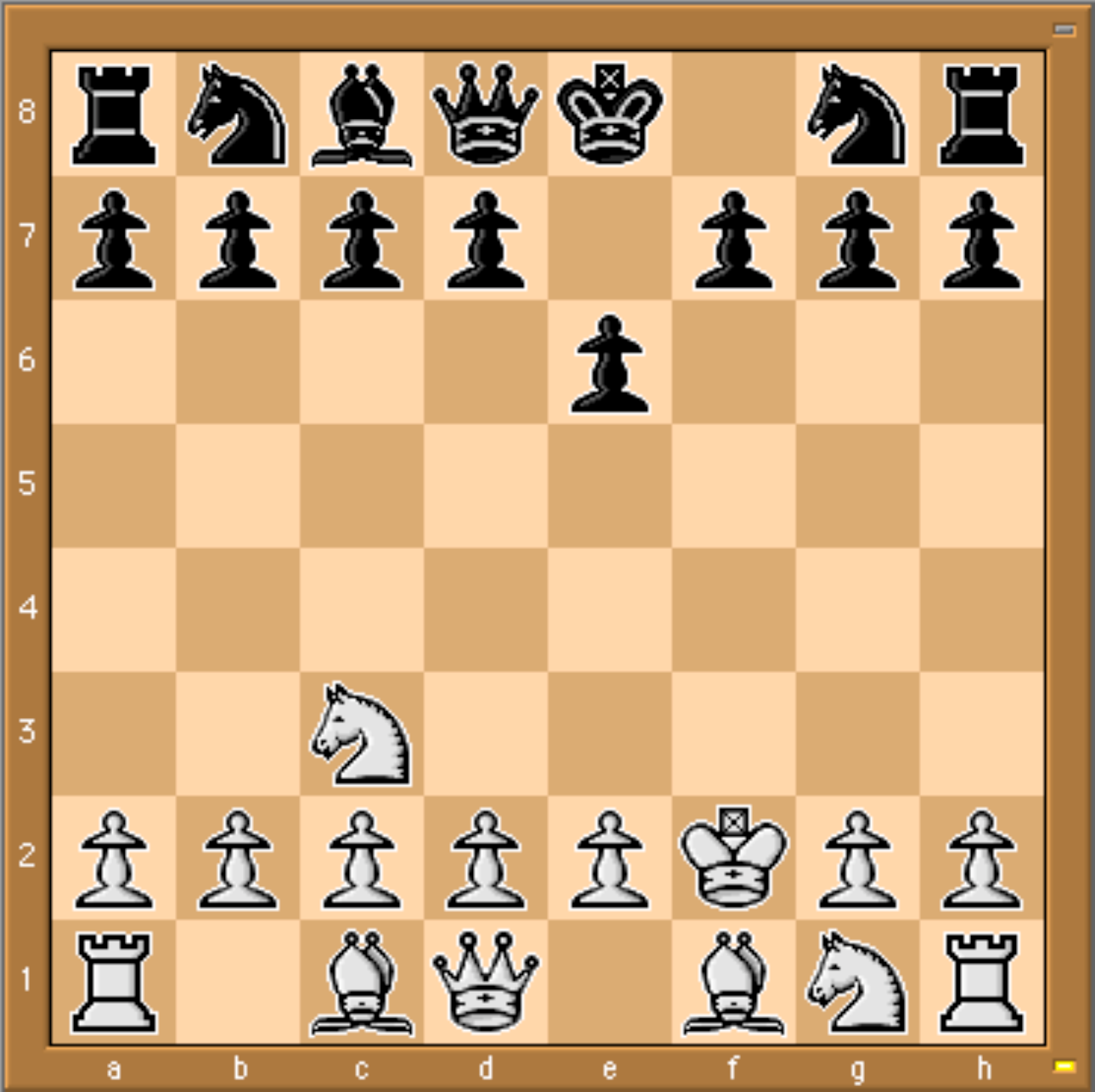}}
\end{center}
\vspace{5pt} 

\noindent
\emph{Conclusion.}  Now Bob has a material advantage, but Alice has one hundred and thirty chips, plus the tie-breaking advantage.  Pondering the board, Bob realizes that if Alice wins the bid for less than thirty, then she can move her queen out to f3 to threaten his king, and then bid everything to win the next move and take his king.  So Bob bids thirty, winning over Alice's bid of twenty-five.  Bob moves his knight to f6, to block the f-column, but Alice can still threaten his king by moving her queen to h4.  Since Alice has enough chips so that she can now win the next two bids, regardless of what Bob bids, and capture the king.  Alice suppresses a smile as Bob realizes he has been defeated.  Head in his hands, he mumbles, \emph{``That was a total mindf**k."}

\vspace{10 pt}

\noindent \textbf{Acknowledgments.}  I am grateful to the organizers and audience at the thirteenth BAD Math Day, in Fall 2006 at MSRI, where discrete bidding games first met the general public, for their patience and warm reception.  I also thank Elwyn Berlekamp, David Eisenbud, and Ravi Vakil for their encouragement, which helped bring this project to completion.  And finally, I throw down my glove at bidding game masters Andrew Ain, Allen Clement, and Ed Finn.  Anytime.  Anywhere.  \emph{---SP}

\section{Preliminaries}

\subsection{Game model}  

Let $G$ be a game played by two players, Alice and Bob, and modeled by a colored directed graph.  The vertices of the graph represent possible positions in the game, and there is a distinguished vertex representing the starting position.  The colored directed edges represent valid moves.  Red and blue edges represent valid moves for Alice and for Bob, respectively, and two vertices may be connected by any combination of red and blue edges, in both directions.  Each terminal vertex represents a possible ending position of the game, and is colored red or blue if it is a winning position for Alice or Bob, respectively, and is uncolored if it is a tie.  For any possible position $v$ in the game, we write $G_v$ for the game played starting from $v$.

Recall that we write $\Gbar$ for the game that is exactly like $G$ except that Alice and Bob exchange roles.  So $\overline G$ is modeled by the same graph as $G$, but with all colors and outcomes switched.

\subsection{Bidding}  Alice and Bob each start with a collection of bidding chips, and all bidding chips have equal value, for simplicity.  When the game begins, the players write down nonnegative integer bids for the first move, not greater than the number of chips in their respective piles.  The bids are revealed simultaneously, and the player making the higher bid gives that many chips to the other, and decides who makes the first move.  The chosen player makes a move in the game, and then the process repeats, until the game reaches an end position or one player is unable to continue.

\subsection{Tie-breaking} \label{tie-breaking section} One player starts, by mutual agreement, with the \emph{tie-breaking advantage}.  If Alice has the tie-breaking advantage, and the bids are tied, then she can either declare Bob the winner of the bid and keep the tie-breaking advantage, or she can declare herself the winner of the bid and give the tie-breaking advantage to Bob.  Similarly, if Bob has the tie-breaking advantage, then he can either declare Alice the winner of the bid, or he can declare himself the winner of the bid and give the tie-breaking advantage to Alice.  In each case, the winner of the bid gives the amount of the bid to the other, and decides who makes the next move.  

One virtue of this tie-breaking method is that it is never a disadvantage to have the tie-breaking advantage (see Lemma~\ref{* is an advantage} below).  Another virtue is that the tie-breaking advantage is worth less than an ordinary bidding chip (Lemma~\ref{1>*}).  Other reasonable tie-breaking methods are possible, and many of the results in this paper hold with other methods. We discuss some other tie-breaking methods in the Appendix.

We write $G(a^*,b)$ for the bidding game in which Alice starts with $a$ bidding chips and the tie-breaking advantage, and Bob starts with $b$ bidding chips.  Similarly, $G(a,b^*)$ is the bidding game in which Alice starts with $a$ bidding chips and Bob starts with $b$ bidding chips and the tie-breaking advantage.

\section{General theory}\label{general theory}

As we make the transition from recreational play to mathematical investigation, one of the most basic questions we can ask about a game $G$ is for which values of $a$ and $b$ does Alice have a winning strategy for $G(a^*,b)$ or for $G(a,b^*)$.  Often it is convenient to fix the total number of chips, and simply ask how many chips Alice needs to win.  And when Alice does have a winning strategy, we ask how to to find it.  The general theory that we present here shows some of the structure that the answers to these questions must have.  For instance, if Alice has a winning strategy for $G(a^*,b)$, then she also has a winning strategy for $G((a+1)^*,b)$, and if she has a winning strategy for $G(a^*,b+1)$, then she also has a winning strategy for $G(a+1, b^*)$.  Similar results hold when Bob starts with the tie-breaking advantage.  In each case, she can just play as if her extra chip wasn't there, or as if Bob's missing chip wasn't missing.  Some of the structural results that we present here, such as the periodicity result in Section \ref{periodicity section} are less obvious, and can be used to greatly simplify computations for specific games.   We apply this approach to solve several games, including Tic-Tac-Toe, in Sections \ref{examples section} and \ref{TTT section}.

For simplicity, we always assume optimal play, and say that Alice wins if she has a winning strategy, and that she does not win if Bob has a strategy to prevent her from winning.  

\subsection{Value of the tie-breaking advantage} \label{basic lemmas}

Roughly speaking, we show that the value of the tie-breaking advantage is strictly positive, but less than that of an ordinary bidding chip.

\begin{lemma} \label{* is an advantage}
If Alice wins $G(a,b^*)$, then she also wins $G(a^*,b)$.
\end{lemma}

\begin{proof}
Alice's winning strategy for $G(a^*,b)$ is as follows.  She plays as if she did not have the tie-breaking chip until the first time the bids are tied.  The first time the bids are tied, Alice declares herself the winner of the bid, and gives Bob the tie-breaking advantage.  The resulting situation is the same as if Bob had started with the tie-breaking advantage and declared Alice the winner of the bid.  Therefore, Alice has a winning strategy for the resulting situation, by assumption.
\end{proof}

\begin{lemma}\label{1>*}
If Alice wins $G(a^*,b+1)$, then she also wins $G(a+1,b^*)$.
\end{lemma}

\begin{proof}
Alice's winning strategy for $G(a+1, b^*)$ is as follows.  She begins by playing as if she started with $a$ chips and the tie-breaking advantage except that whenever her strategy for $G(a^*, b+1)$ called for bidding $k$ and using the tie-breaking advantage, she bids $k+1$ instead.  She continues in this way until either she wins such a bid for $k+1$ or Bob uses the tie-breaking advantage.

Suppose that Alice's strategy for $G(a^*, b+1)$ called for bidding $k$ for the first move and using the tie-breaking advantage in case of a tie.  Then Alice bids $k+1$ for the first move.  If Alice wins the bid, then the resulting situation is the same as if Alice had won the first bid in $G(a^*, b+1)$ using the tie-breaking advantage, so Alice has a winning strategy by hypothesis.  Similarly, if Bob wins the bid using the tie-breaking advantage, then the resulting situation is the same as if Bob had won the first bid in $G(a^*, b+1)$ by bidding $k+1$, so Alice has a winning strategy.  Finally, if Bob bids $k+2$ or more chips to win the bid, then the resulting situation is a position that could have been reached following Alice's winning strategy for $G(a+1, b^*)$, except that Alice has traded the tie-breaking advantage for two or more chips, and Alice can continue with her modified strategy outlined above.

The analysis of the case where Alice's strategy for $G(a^*,b+1)$ did not call for using the tie-breaking advantage for the first move is similar.
\end{proof}

Although Lemma \ref{1>*} shows that trading the tie-breaking chip for a bidding chip is always advantageous, giving away the tie-breaking chip in exchange for an extra bidding chip from a third party is not necessarily a good idea; for any positive integer $n$, there is a game $G$ such that Alice has a winning strategy for $G(a^*,b)$, but not for $G(a+n, b^*)$, as the following example demonstrates.

\begin{example} \label{bid zero example}
Let $G$ be the game where Bob wins if he gets any of the next $n$ moves, and Alice wins otherwise. Then for $n \ge 1$, $(k^*, 0)$ is an Alice win if and only if $k\ge 2^{n-1}-1$, while $(k, 0^*)$ is an Alice win if and only if $k\ge 2^n-1$. 
\end{example}

\subsection{Using the tie-breaking advantage}

In order for the tie-breaking advantage to have strictly positive value, as shown in Lemma~\ref{* is an advantage}, it is essential that the player who has it is not required to use it.  However, the following proposition shows that it is always a good idea to use the tie-breaking advantage, unless you want to bid zero.

\begin{proposition} \label{use the advantage}
Both players have optimal strategies in which they use the tie-breaking advantage whenever the bids are nonzero and tied.
\end{proposition}

\begin{proof}
Suppose that Alice has an optimal strategy which involves bidding $k$, but letting Bob win the bid if the bids are tied. If $k$ is positive, then Alice can do at least as well by bidding $(k-1)^*$ instead.  If Bob bids $k$ or more, the resulting situation is unchanged, while if Bob bids $k-1$ or less, then Alice pays $(k-1)^*$ instead of $k$ to win the bid, which is at least as good by Lemma~\ref{1>*}.
\end{proof}

If the bids are tied at zero, it is not necessarily a good idea to use the tie-breaking advantage, as the following example shows.

\begin{example}
Consider the game where the player who makes the second move wins.  Suppose Alice and Bob are playing this game, and they both start with the same number of chips.  Then the player who starts with the tie-breaking advantage has a unique winning strategy---bid zero for the first move, decline to use the tie-breaking advantage, and bid everything to make the second move and win.
\end{example}

Proposition~\ref{use the advantage} shows that when looking for an optimal strategy, we can always assume that the player with the tie-breaking advantage either bids 0 or $ 0^*, 1^*, \ldots$.   Furthermore, if the player with the tie-breaking advantage bids 0, then the second player wins automatically and does best to bid 0 as well. Otherwise, if the player with the tie-breaking advantage bids $k^*$, we may assume that the second player either bids $k$ and gains $k^*$ chips while letting the first player move, or else bids $k+1$ and wins the bid.  These observations significantly reduce the number of bids one needs to consider when searching for a winning strategy.

\subsection{Classical Richman calculus}

For the reader's convenience, here we briefly recall the classical methods for determining the critical threshold $R(G)$ between zero and one such that Alice has a winning strategy if her proportion of the bidding resources is greater than $R(G)$ and does not have a winning strategy if her proportion of the bidding resources is less than $R(G)$.  This \emph{Richman calculus} also gives a method for finding the optimal moves and optimal bids for playing $G$ as a bidding game with real-valued bidding.  See the original papers \cite{LLPU, LLPSU} for further details.  In Section \ref{discrete section}, we present a similar method for determining the number of chips that Alice needs to win a discrete bidding game with a fixed total number of chips, as well as the optimal bids and moves for discrete bidding.

First, suppose $G$ is bounded.  We compute the critical thresholds $R(G_v)$ for all positions $v$ in $G$ by working backwards from the end positions.  If $v$ is an end position then
\[
R(G_v) = \left\{ \begin{array}{ll} 0 & \mbox{ if } v \mbox{ is a winning position for Alice.} \\
                                                        1 & \mbox{ otherwise.}
\end{array} \right.
\]                                                        
Suppose $v$ is not an end position.  If Alice makes the next move, then she will move to a position $w$ such that $R(G_w)$ is minimal.  Similarly, if Bob makes the next move, then he will move to a position $w'$ such that $R(G_{w'})$ is maximal.  We define
\[
R_A(G_v) = \min_{A: v \rightarrow w} R(G_w) \mbox{ \ \ and \ \ } R_B(G_v) = \max_{B: v \rightarrow w'} R(G_{w'}),
\]
where the minimum and maximum are taken over Alice's legal moves from $v$ and Bob's legal moves from $v$, respectively.  The critical threshold $R(G_v)$ is then
\[
R(G_v) = \frac{R_A(G_v) + R_B(G_v)}{2}.
\]
The difference $R_B(G_v) - R_A(G_v)$ is a measure of how much both players want to move (or to prevent the other player from moving).  If this difference is positive, then both players want to move, and if the difference is negative then the position is zugzwang and both players want to force the other to move.  In either case, an optimal bid for both players is $\Delta_v = |R_B(G_v) - R_A(G_v)| / 2$.

Next, suppose $G$ is locally finite, but not necessarily bounded.  Let $G[n]$ be the truncation of $G$ after $n$ moves.  So $G[n]$ is just like $G$ except that the game ends in a tie if there is no winner after $n$ moves.  In particular, Alice wins $G[n]$ if and only if she has a strategy to win $G$ in at most $n$ moves.  We can compute the critical threshold $R(G)$ when $G$ is bounded by using the bounded truncations $G[n]$, as follows.  First, $R(G[n])$ can be computed by working backward from end positions, since $G[n]$ is bounded.  Now $R(G[n])$ is a nonincreasing function of $n$ that is bounded below by zero, so these critical thresholds approach a limit as $n$ goes to infinity.  Furthermore, since $G$ is locally finite, Alice has a winning strategy for $G$ if and only if she has a winning strategy that is guaranteed to succeed in some fixed finite number of moves.  It follows that
\[
R(G) = \lim_{n \rightarrow \infty} R(G[n]).
\]

For games that are not locally finite, Alice may have a winning strategy, but no strategy that is guaranteed to win in a fixed finite number of moves.  In this case, $R(G)$ is not necessarily the limit of the critical thresholds $R(G[n])$, as the following example shows.

\vspace{10 pt}

\begin{example} 
Let $\A^m$ be the game that Alice wins after $m$ moves, and let $G$ be the game in which the first player to move can choose between the starting positions of $\A^m$ for all positive integers $m$.  Then Alice is guaranteed to win $G$, so $R(G) = 0$, but the critical threshold of each truncation is $R(G[n]) = 1/2$.  Indeed, if Bob wins the first move of $G[n]$, then he can move to the starting position of $\A^n$, which Alice cannot win in the remaining $n-1$ moves.
\end{example}

\subsection{Discrete Richman calculus}\label{discrete section}

Here we return to discrete bidding and compute the number of chips that Alice needs to win a locally finite game, assuming that the total number of chips is fixed.  Since Alice may or may not have the tie-breaking advantage, the total number of chips that Alice has is an element of $\N \cup \N^*$, which is totally ordered by
\[
0 < 0^* < 1< 1^* < 2 < \cdots.
\]
If we fix the game $G$ and the total number of ordinary chips $k$, then it follows from Lemmas~\ref{* is an advantage} and \ref{1>*} that there is a critical threshold $f(G,k) \in \N \cup \N^*$ such that Alice wins if and only if she has at least $f(G,k)$ chips.  Note that Alice can have at most $k^*$ chips, so if $G$ is a game in which Alice never wins, then $f(G,k) = k+1,$ by definition.

The critical thresholds $f(G,k)$ can be computed recursively from end positions for bounded games, and the critical thresholds of locally finite games can be computed from the critical thresholds of their truncations, just like the critical thresholds $R(G)$ for real-valued bidding.  However, one must account for the effects of rounding, since the bidding chips are discrete, as well as the tie-breaking advantage.  

First, suppose that $v$ is an end position.  Then
\[
f(G_v,k) = \left\{ \begin{array}{ll} 0 & \mbox{ if } v \mbox{ is a winning position for Alice.} \\
                                                        k+1 & \mbox{ otherwise.}
\end{array} \right.
\]                                                        
Next, suppose $v$ is not an end position.  If Alice makes the next move, then she will move to a position $w$ such that $f(G_w,k)$ is minimal.  Similarly, if Bob makes the next move, then he will move to a position $w'$ such that $f(G_{w'},k)$ is maximal.  We define
\[
f_A(G_v,k) = \min_{A: v \rightarrow w} f(G_w,k) \mbox{ \ \ and \ \ } f_B(G_v,k) = \max_{B: v \rightarrow w'} f(G_{w'},k),
\]
where the minimum and maximum are taken over Alice's legal moves from $v$ and Bob's legal moves from $v$, respectively.  

For an element $x \in \N \cup \N^*$, we write $|x|$ for the underlying integer, so $|a|$ and $|a^*|$ are both equal to $a$, for nonnegative integers $a$.  We also define $a + *  = a^*$.  For a real number $x$, we write $\lfloor x \rfloor$ for the greatest integer less than or equal to $x$.

\begin{theorem} \label{discrete recursion}
For any position $v$, the critical threshold $f(G_v,k)$ is given by
\[
f(G_v,k) = \left \lfloor \frac{ | f_A(G_v, k) | +  | f_B(G_{v},k) | }{ 2 } \right \rfloor + \varepsilon,
\]
where 
\[
\varepsilon = \left \{ \begin{array}{ll} 0 & \mbox{ if } |f_A(G_v,k)| + |f_B(G_{v},k)| \mbox{ is even, and } f_A(G_v, k) \in \N. \\
								    1 & \mbox{ if } |f_A(G_v,k)| + |f_B(G_{v},k)| \mbox{ is odd, and } f_A(G_v, k) \in \N^*. \\
								    * & \mbox{ otherwise.} \\
								    \end{array} \right.
\]
\end{theorem}

\begin{proof}
Since the critical threshold for any locally finite game can be computed from its bounded truncations, it is enough to prove the theorem in the case where $G$ is bounded.  If the game starts at an end position, then the theorem is vacuously true.  We proceed by induction on the length of the bounded game.

Suppose $|f_A(G_v,k)| + |f_B(G_{v},k|$ is even and $f_A(G_v,k) \in \N$.  If Alice has at least $(|f_A(G_v,k)| + |f_B(G_{v},k|) / 2$ chips, then she can bid
\[
\Delta = \big| |f_A(G_v,k)| - |f_B(G_{v},k| \big| /2
\]
and guarantee that she will end up in a position $w$ with at least $f(G_w,k)$ chips.  Then Alice has a winning strategy, by induction, since $G_w$ is a bounded game of shorter length than $G$.  Similarly, if Alice has fewer than $( |f_A(G_v,k)| + |f_B(G_{v},k|)  /2$ chips, then Bob can bid $\Delta$ and guarantee that he will end up in a position $w'$ where Alice will have fewer than $f(G_{w'},k)$ chips.  Then Bob can prevent Alice from winning, by induction.  

Therefore, Alice wins $G$ if and only if she has at least $(|f_A(G_v,k)| + |f_B(G_{v},k|)/2$ chips, as was to be shown.  The proofs of the remaining cases, when $|f_A(G_v,k)| + |f_B(G_{v},k|$ is odd, and when $f_A(G_v,k)$ is in $\N^*$, are similar.  If $|f_A(G_v,k)| + |f_B(G_v,k)|$ is odd and $f_A(G_v,k)$ is in $\N^*$, then the ideal bid for both players is the round down
\[
\Delta = \Big \lfloor \big| |f_A(G_v,k)| - |f_B(G_{v},k| \big| /2 \Big \rfloor.
\]
If $|f_A(G_v,k)| + |f_B(G_v,k)|$ is odd but $f_A(G_v,k)$ is in $\N$, then both players should try to make the smallest possible bid that is strictly greater than $ \Big \lfloor \big| |f_A(G_v,k)| - |f_B(G_{v},k| \big|  /2 \Big \rfloor$.  And if $|f_A(G_v,k)| + |f_B(G_v,k)|$ is even and $f_A(G_v,k)$ is in $\N^*$ then both players should try to make the smallest possible bid that is strictly greater than $\big| |f_A(G_v,k)| - |f_B(G_{v},k|  /2 \big| -1$.  
\end{proof}

\noindent Theorem~\ref{discrete recursion} makes it possible to find both the critical threshold and the optimal strategy for a given bounded game by working backward from end positions.

\begin{example} \label{simple tables}
Suppose $\A$ is a game that Alice is guaranteed to win and $\B$ is a game that Bob wins.  Then
\[
f(\A,k) = 0 \mbox{ \ \ and \ \ } f(\B,k) = k+1.
\]
Let $\E$ be the game in which the first player to move wins.  Then
\[
f(\E,k) = \left \{ \begin{array}{ll} (k+1)/2 & \mbox{ if } k \mbox{ is odd}. \\
					        \lfloor (k+1)/2 \rfloor ^* & \mbox{ if } k \mbox{ is even}.
					        \end{array} \right.
\]
As games become more complicated, it is more convenient to encode the possibilities in a table.  For instance, the critical thresholds for the game $\E$ could be put in a table as follows.

\begin{center}
\vskip 10 pt \begin{tabular}{|r||*{2}{l|}}\hline 
$k=2n+$&0&1 \\ \hline \hline
$f(\E, k) = n + $&0*&1 \\ \hline 
\end{tabular}
\end{center}
\vskip 10pt

\noindent Let $A^2$ be the game that Alice wins if she makes either of the first two moves and Bob wins otherwise.  Similarly, let $B^2$ be the game that Bob wins if he makes either of the first two moves and Alice wins otherwise.  Then the critical thresholds for $A^2$ and $B^2$ are given by

\begin{center}
\vskip 10 pt \begin{tabular}{|r||*{4}{l|}}\hline 
$k=4n+$&0&1&2&3\\ \hline \hline
$f(A^2, k) = n + $&0&0*&0*&1\\ \hline 
\end{tabular}
\mbox{ \ \ and \ \ }
\begin{tabular}{|r||*{4}{l|}}\hline 
$k=4n+$&0&1&2&3\\ \hline \hline
$f(B^2, k) = 3n + $&1&1*&2*&3\\ \hline 
\end{tabular}
\end{center}
\vskip 10pt

\noindent See Section~\ref{TTT section} for detailed computations using such tables in a more interesting situation.
\end{example}

\subsection{Discrete bidding with large numbers of chips}

When $G$ is played as a discrete bidding game, the optimal moves for Richman play are not necessarily still optimal.  This may be seen as a consequence of the effects of rounding and tie-breaking in the discrete Richman calculus.  However, one still expects that as the number of chips becomes large, discrete bidding games should become more and more like Richman games.  Roughly speaking, the effects of rounding should only be significant enough to affect the outcome when the number of chips is small or Alice's proportion of the total number of chips is close to the critical threshold $R(G)$.  We think of these situations as ``unstable."

\begin{definition}  
A strategy for Alice is \emph{stable} if, whenever Alice makes a move following this strategy, she moves to a position $w$ such that $R(G_w)$ is as small as possible.  
\end{definition}

\noindent Similarly, we say that a strategy for Bob is stable if, whenever he makes a move following this strategy, he moves to a position $w'$ such that $R(G_{w'})$ is as large as possible.  We say that a discrete bidding game is stable if both Alice and Bob have stable optimal strategies.  Note that the proofs of Lemmas~\ref{* is an advantage} and \ref{1>*} go through essentially without change when ``Alice wins" is replaced by ``Alice has a stable winning strategy."  For instance, if Alice has a stable winning strategy for $G(a,b^*)$, then she also has a stable winning strategy for $G(a^*,b)$.

\begin{theorem} \label{above the threshold}
For any locally finite game $G$, and for any positive $\epsilon$, Alice has a stable winning strategy for $G(a,b^*)$ provided that $a/(a+b)$ is greater than $R(G) + \epsilon$ and $a$ is sufficiently large.
\end{theorem}

\begin{proof}
First, we claim that it suffices to prove the theorem when $G$ is bounded.  Recall that $G[n]$ denotes the truncation of $G$ after $n$ moves.  So $G[n]$ is bounded and Alice wins $G[n]$ if and only if she wins $G$ in at most $n$ moves.  By \cite[Section~2]{LLPU}, $R(G)$ is the limit as $n$ goes to infinity of $R(G[n])$.  Therefore, replacing $G$ by $G[n]$ for $n$ sufficiently large, we may assume that $G$ is bounded.

Suppose $G$ is bounded and guaranteed to end after $n$ moves.  If $n = 1$, then the theorem is clear.  We proceed by induction on $n$.  Assume the theorem holds for games guaranteed to terminate after $n-1$ moves.  Alice's strategy is as follows.  For the first move, she bids $x$ such that $x/(a+b)$ is between the optimal real-valued bid $\Delta$ and $\Delta + \epsilon /2$, which is possible since $a$ is large.  If Alice moves, then she moves to a position $w$ such that $R(G_w)$ is minimal.  Otherwise, Bob moves wherever he chooses.  Either way, Alice ends up in a game $G_v$ that is guaranteed to terminate after $n-1$ moves, holding $a'$ chips where $a' / (a+b)$ is greater than $R(G_v) + \epsilon / 2$, and hence has a winning strategy for $a$ sufficiently large, by induction.  Since $G$ is locally finite, there are only finitely many possibilities for $v$.  Therefore, $a$ can be chosen sufficiently large for all such possibilities, and the result follows.
\end{proof}

\noindent The conclusion of Theorem~\ref{above the threshold} is false if $G$ is not locally finite; we give an example illustrating this in Section~\ref{examples section}.

\begin{theorem} \label{below the threshold}
For any finite game $G$, and for any positive $\epsilon$, Bob has a stable strategy for preventing Alice from winning $G(a^*,b)$ provided that $a/(a+b)$ is less than $R(G) - \epsilon$ and $b$ is sufficiently large.
\end{theorem}

\begin{proof}
Recall that $\overline G$ is the game that is identical to $G$ except that Alice and Bob exchange roles.  If $G$ is finite, then $R(\overline G) = 1- R(G)$.  Therefore, the result follows from Theorem~\ref{above the threshold} applied to $\overline G(b,a^*)$.
\end{proof}

\noindent  Under the hypotheses of Theorem~\ref{below the threshold}, Bob actually has a winning strategy.  For locally finite games that are not finite, $R(\overline G)$ may be strictly larger than $1-R(G)$, so Bob should not be expected to have a winning strategy for $G(a^*,b)$.  With real-valued bidding he may only have a strategy to prolong the game into an infinite draw.

Next, we show that finite games always become stable when the number of chips becomes sufficiently large.

\begin{theorem}\label{stable for many chips}
Let $G$ be a finite game.  Then $G(a^*,b)$ and $G(a,b^*)$ are stable when $a+b$ is sufficiently large.
\end{theorem}

\begin{proof}
Assume the total number of chips is large.  We will show that either 
\begin{enumerate}
\item Alice has a stable winning strategy, or
\item Any unstable strategy for Alice can be defeated by a stable strategy for Bob.
\end{enumerate}
A similar argument shows that either Bob has a stable winning strategy, or any unstable strategy for Bob can be defeated by a stable winning strategy for Alice, and the theorem follows.

Suppose Alice follows an unstable strategy that calls for her to move to a position $w$ such that $R(G_w)$ is not as small as possible.  Let $\delta$ be the discrepancy
\[
\delta = R(G_w) - R(G_{w_0}),
\]
where $w_0$ is a position that Alice could have moved to such that $R(G_{w_0})$ is as small as possible.  If Alice's proportion of the bidding chips is greater than $R(G) + \delta/4$, then she has a stable winning strategy, by Theorem~\ref{above the threshold}.  Therefore, we may assume that Alice's proportion of the bidding chips is at most $R(G) + \delta/4$.

Suppose that the position is not zugzwang, so Bob is bidding for the right to move.  Since the total number of chips is large, Bob can make a bid that is strictly between $\Delta - \delta/2$ and $\Delta - \delta/4$, where $\Delta = R(G) - R(G_{w_0})$ is the optimal real-valued bid.  Then, if Alice wins the bid and moves to $w$, she finds herself with a proportion of chips that is less than $R(G_w) - \delta/4$ and hence Bob has a stable winning strategy (since the number of chips is large).  Otherwise, Bob wins and moves to a position $w'$ such that $R(G_{w'})$ is as large as possible.  Then his proportion of the chips is greater than $R(\overline G_{w'}) + \delta/4$, so again he has a stable winning strategy.

The proof when the position is zugzwang is similar except that Bob should bid between $\Delta + \delta/4$ and $\Delta + \delta/2$ to force Alice to move.
\end{proof}

\subsection{Periodicity} \label{periodicity section}

Here we prove a periodicity result for finite stable games that allows one to determine the outcome of $G$ for all possible chip counts for Alice and Bob by checking only a finite number of cases.  We use this result extensively in our analysis of specific bidding games in Sections \ref{examples section} and \ref{TTT section}. 

Fix a finite game $G$.  Choose a positive integer $M$ such that $M \cdot R(G_v)$ and $M \cdot \Delta_v$ are integers for all positions $v$ in $G$.  For instance, one can take $M$ to be the least common denominators of $R(G_v)$ and $\Delta_v$ for all $v$.  Let 
\[
m = M \cdot R(G) \mbox{ \ \  and \ \ } m_v = M \cdot R(G_v).
\]
Similarly, let $\overline m = R(\overline G)$ and $\overline m_v = R(\overline G_v)$.

\begin{theorem}  \label{periodicity for Alice}
If Alice has a stable winning strategy for $G(a^*,b)$ then she also has a stable winning strategy for $G(a + m^*, b + \overline m)$.  Similarly, if Alice has a stable winning strategy for $G(a,b^*)$ then she also has a stable winning strategy for $G(a + m, b + \overline m^*)$.
\end{theorem}

\begin{proof}
Since $G$ is finite, Alice's stable winning strategy for $G(a^*,b)$ is guaranteed to succeed in some fixed number of moves. If the game starts at a winning position for Alice, then the theorem is vacuously true, so we proceed by induction on the number of moves.

Suppose Alice has a stable winning strategy for $G(a^*, b)$ in which she bids $k$ for the first move.  Then Alice can win $G(a + m^*, b+ \overline m)$ by bidding $k + M \Delta$ for the first move, and moving according to her stable strategy for $G(a^*,b)$.
Regardless of whether she wins the bid, Alice ends up in a position $v$ where, compared to her stable strategy for $G(a^*,b)$, she has at least  $m_v$ additional chips and Bob has at most $\overline m_v$ additional chips.  By induction, Alice has a stable winning strategy starting from $v$ that is guaranteed to win in a smaller number of moves, and the result follows.

The proof for the situation where Bob starts with tie-breaking advantage is similar.
\end{proof}

\begin{theorem} \label{periodicity for Bob}
If Bob has a stable strategy to prevent Alice from winning $G(a^*,b)$ then he also has a stable strategy to prevent Alice from winning $G(a + m^*, b + \overline m)$.  Similarly if Bob has a stable strategy to prevent Alice from winning $G(a,b^*)$, then he also has a stable strategy to prevent Alice from winning $G(a + m, b + \overline m^*)$.
\end{theorem}

\begin{proof}
Similar to proof of Theorem~\ref{periodicity for Alice}.
\end{proof}

Using these periodicity results, we can determine the exact set of chip counts for which Alice can win $G$ by answering the following two questions for finitely many $x$ and $y$ in $\N \cup \N^*$.
\begin{itemize}
\item If Alice starts with $x$ chips then how many chips does Bob need to prevent her from winning?
\item If Bob starts with $y$ chips then how many chips does Alice need to win?
\end{itemize}
There is a unique minimal answer in $\N \cup \N^*$ to each such question, and the answer is generally not difficult to find if $G$ is relatively simple and $x$ or $y$ is small.  Furthermore, by Theorem~\ref{stable for many chips}, there is an integer $n$ such that $G(a^*,b)$ and $G(a,b^*)$ are stable provided that $a + b$ is at least $n$.  Then, if we know the answers to the above questions for $x < m + n$ and $y < \overline m + n$, we can easily deduce whether Alice wins for any given chip counts using Theorems~\ref{periodicity for Alice} and \ref{periodicity for Bob}.

\vspace{5 pt}

We conclude this section with some open problems that ask to what extent, if any, these results extend from finite games to locally finite games.  For fixed $\epsilon > 0$ and a large number of chips, by Theorem~\ref{above the threshold} we know that Alice has a stable winning strategy if her chip count is at least $R(G) + \epsilon$ and Bob has a winning strategy if his proportion of the chips is at least $R(\overline G) + \epsilon$.  However, if $R(G) + \epsilon$ is less than $1 - R(\overline G) - \epsilon$, then there is a gap where the outcome is unclear, even when the number of chips is large.

\begin{problem} \label{in the draw range}
Is there a locally finite game $G$ and a positive number $\epsilon$ such that Alice has a winning strategy for infinitely many chip counts $G(a^*,b)$ such that $a/(a+b)$ is less than $R(G) - \epsilon$?
\end{problem}

\noindent Roughly speaking, Problem~\ref{in the draw range} asks whether strategies to force a draw in a locally finite game with real-valued bidding can always be approximated sufficiently well by discrete bidding with sufficiently many chips.  However, it is not clear whether one should follow stable strategies in locally finite games with large numbers of chips.

\begin{problem}\label{eventually Richman}
If $G$ is locally finite, are $G(a^*,b)$ and $G(a,b^*)$ stable for $a$ and $b$ sufficiently large?
\end{problem}

If the answer to Problem~\ref{in the draw range} is negative, then the answer to Problem~\ref{eventually Richman} is negative as well.  To see this, suppose Alice wins a game $G$ with a proportion of chips less than $R(G)-\epsilon$ and an arbitrarily large total number of chips.  Let $H$ be a stable game with Richman value between $R(G)-\epsilon$ and $R(G)$ (which is not difficult to construct), and let $G \wedge H$ be the game in which the first player to move gets to choose between the starting position of $G$ and the starting position of $H$.   In $G \wedge H$, Alice's optimal first move for large chip counts would be to move to the starting position for $G$, while her stable strategy (i.e.\ her optimal strategy for real-valued bidding) would be to move to the starting position for $H$.  In the Appendix, we show that the answer to Problem~\ref{eventually Richman} is negative for different tie-breaking methods.  

\section{Examples}\label{examples section}

In this section, we analyze discrete bidding play for two simple combinatorial games, Tug o' War and what we call \emph{Ultimatum}.   We use these examples to construct games with strange behavior under discrete bidding play.

\subsection{Tug o' War} \label{tug o war}

The Tug o' War game of length $n$, which we denote by $\Tug^n$, is played on a path of length $2n$, with vertices labeled $-n, ..., -1, 0 , 1, ..., n$, from left to right. The game starts at the center vertex, which is marked $0$.  Adjacent vertices are connected by edges of both colors in both 
directions.  Alice's winning position is the right-most vertex of the path, and Bob's winning position is the left-most vertex of the path.  So Alice tries to move to the right, Bob tries to move the left, and the 
winner is the first player to reach the end.  In particular, Tug o' War is stable, since the optimal moves do not depend on whether bidding is discrete or real-valued, and since it is also symmetric the critical threshold is $R(\Tug^n) = 1/2$.  

\begin{proposition}  \label{tug o war result}
Suppose Bob's total number of chips is less than $n$.  Then Alice wins $\Tug^n$ if and only if her total number of chips is at least $(n-1)^*$.
\end{proposition}

\begin{proof}
We define the weight of a position in the bidding game to be the number of the current vertex plus the number of chips that Alice has, including the tie-breaking chip.  Note that, in order to reach a winning position, Alice must first reach a position of weight at least $n$.

Suppose Alice's chip total is at most $n-1$.  Then Bob can force a draw by bidding zero every time, and using the tie-breaking advantage to win whenever possible.  Indeed, if Bob does this, then the weight of the position in the game never exceeds its starting value, so Alice cannot win.

Suppose Alice's chip total is at least $(n-1)^*$.  Then Alice can win with the following strategy.  Whenever she has the tie-breaking advantage, she bids zero and uses the advantage if possible.  Whenever she does not have the tie-breaking advantage, she bids one.  With this strategy, the weight of the position never decreases, and it follows that Bob cannot win.  Bob may win the first several moves, but eventually he will run out of chips, and Alice will win a move for zero, using the tie-breaking advantage.  Then Alice may win a certain number of moves for one chip each.  Since the weight of the position is at least $n$, if Bob lets her win moves for one chip indefinitely, then Alice will win.  So eventually Bob must bid either two chips or one plus the tie-breaking advantage, and the weight of Alice's position increases by one.  It follows that Alice can raise the weight of her position indefinitely, until eventually she must win. 
\end{proof}

The Richman game version of Tug o' War was studied in \cite[p.\ 252]{LLPSU}, where the critical threshold of the vertex labeled $k$ was determined to be $(k + n) / 2n$.  Therefore, the periodicity results of Section~\ref{periodicity section} hold with $M = 2n$ and with $m$ and $\overline m$ both equal to $n$.  The cases covered by Proposition~\ref{tug o war result} then completely determine the outcome of Tug o' War for all possible chip counts. 

\begin{corollary}
Let $a$, $b$, $k$, and $k'$ be nonnegative integers, with $a$ and $b$ less than $n$.  Then Alice wins $\Tug^n(kn + a, (k'n + b)^*)$ if and only if $k$ is greater than $k'$.  Furthermore, Alice wins $\Tug^n(kn + a^*, k'n + b)$ if and only if either $k$ is greater than $k'$ or $k$ is equal to $k'$ and $a$ is equal to $n-1$.
\end{corollary}

Tug o' War game is perhaps the simplest game that is not bounded, and yet we can use it to construct some interesting examples of bidding game phenomena for games that are not locally finite.

\begin{proposition}
Let $G$ be the game in which the first player to move can go to the starting position of $\Tug^n$ for any $n$.  Then $R(G) = 1/2$, but $G(3a, a^*)$ is a draw for all values of $a$ greater than one.
\end{proposition}

\begin{proof}
Since all of the possible first moves lead to positions $v$ with $R(G_v) = 1/2$, the critical threshold is $R(G) = 1/2$.   Consider $G(3a, a^*)$ for any $a$.  Bob's strategy to force a draw is as follows.  Bob bets all of his chips for the first move. Then Alice can either bet $a+1$ and win the bet, getting to play any $\Tug^n(2a-1, 2a+1^*)$ which is at best a draw for Alice. Otherwise, Bob wins the bet, and chooses to play $\Tug^n$ for some $n$ greater than $4a + 1$, which leads to a draw by Proposition~\ref{tug o war result}. So Bob can force a draw.
\end{proof}

\noindent This type of behavior, where Bob can force a draw even though Alice's proportion of the chips is much greater than $R(G)$ and the total number of chips is large, is impossible for locally finite games by Theorem~\ref{above the threshold}.

The following example shows that non locally finite games may also be unstable even for large numbers of chips.

\begin{proposition}
Let $G$ be the game in which the first player to move can either play $\Tug^1$, or $\Tug^n$ starting at the node labeled $-1$ for any $n\ge 10$. Then $G(3a^*, 2a)$ fails to be stable for all $a$.
\end{proposition}

\begin{proof}
We claim that Bob has no optimal stable strategy.  For the first move, Bob's only stable strategy is to move to $\Tug^1$, since its Richman value is 1/2 and the Richman values of all $\Tug^n$ starting at $-1$ is less than 1/2. However, if Bob wins the first bid and moves to $\Tug^1$ then he will lose.  Nevertheless, we claim that Bob has a strategy to force a draw.  This strategy is as follows.

Bob bets all of his chips on the first turn. If Alice lets him win the bid, he can move to $\Tug^n$ for $n$ large, which leads to a draw. Otherwise, if Alice bets $2a^*$, she may choose  to play either $\Tug^1(a,4a^*)$ or $\Tug^n(a,4a^*)$ starting from the node labeled $-1$.  If Alice chooses $\Tug^1$ then Bob wins.  Otherwise, Bob can win the next move for $a^*$, leading to $\Tug^n(2a^*, 3a)$, which is at worst a draw for Bob.  Therefore Bob has a nonstable strategy that is better than any stable strategy.
\end{proof}

\subsection{Ultimatum}

We know describe the Ultimatum game of degree $n$, which we denote by $\Ult^n$.  It is played on a directed graph with vertices labeled $B, -n, \ldots, -1, 0 ,1, \ldots, n-1, n, A$.  There are red edges from $0$ to $n$, from $k$ to $A$ for $k > 0$, and from $k$ to $k + 1$ for $k < 0$.  Similarly, there are blue edges from $0$ to $-n$, from $k$ to $B$ for $k < 0$, and from $k$ to $k-1$ for $k > 0$.  The game starts at $0$.  In other words, when the game starts, the first player to move gives the other an ultimatum---the other player must make each of the next $n$ moves (in which case the game reverts to the beginning state), or else lose the game.

Since $\Ult^n$ is finite and symmetric, the critical threshold $R(U^n)$ is $1/2$.

\begin{proposition}
Suppose $b$ is less than $2^n$.  Then Alice has a winning strategy for $\Ult^n(a^*, b)$ if and only if $a$ is greater than $b$, or $a$ is equal to $b$ and $b\neq 2^{n-1} - 1$.
\end{proposition}

\begin{proof}
Suppose $a$ is greater than $b$.  We claim that Alice can win by bidding $b$ chips on the first move and using the tie-breaking advantage.  Then Alice still has at least one chip left, so Bob must bid at least one chip plus the tie-breaking advantage for the second move, three chips for the third move, and $3 \cdot 2^k$ for move number $k + 3$.  It follows that if Bob is able to make $n$ moves in a row, then Alice receives at least $3 \cdot (2^{n-1} -1)$ chips from Bob before returning to $0$.  In particular, Alice returns to the starting position with strictly more chips than she 
started with, and hence must eventually win.

Suppose $a$ is equal to $b$ and less than $2^{n-1} - 1$.  Then Alice can win by bidding all of her chips on the first move and using the tie-breaking advantage.  Bob must give her the tie-breaking advantage to take the second move, and one chip to take the third move, and $2^k$ chips to take move number $k + 3$.  It follows that if Bob is able to make $n$ moves in a row, then Alice will have collected $2^{n-1} - 1$ chips from Bob by the time they return to the starting position.  Now Alice has more chips than Bob, plus the tie-breaking advantage, so she has a winning strategy.

Suppose $a$ is equal to $b$ and greater than $2^{n-1} -1$. Then Alice bids $a -1$; if Bob bids all of his chips to win, then Alice can pay him $2^{n-1} - 1$ chips, plus the tie-breaking advantage, to return to the starting position.  Now Alice has at least two more chips than Bob, so she can bid Bob's number of chips plus one to move to vertex $n$.  Then, if Bob has enough chips to make the next $n$ moves, Alice will return to the starting vertex with more chips than Bob, plus the tie-breaking advantage, and will therefore win.

Finally, suppose $a$ and $b$ are both equal to $2^{n-1} - 1$.  Then Bob can prevent Alice from winning by bidding all of his chips for the first move.  If Alice bids all of her chips plus the tie-breaking advantage to take the first move, then Bob has exactly enough chips to return to the starting position with $2^{n-1} -1$ chips left.  Otherwise, Bob makes the first move, and Alice has just enough chips to return to the starting position with $2^{n-1} -1 $ chips left.  Since the tie-breaking advantage is always an advantage, by Lemma \ref{* is an advantage}, Bob's position is no worse than when the game began, so Alice cannot win.
\end{proof}

\begin{proposition}
Suppose $b < 2^n$.  Then Alice has a winning strategy for $\Ult^n(a, b^*)$ if and only if $a$ is greater than $b+1$, 
or 
$a$ is equal to $b+1$ and $b\neq 2^{n-1} - 1$.
\end{proposition}

\begin{proof}
The proof is essentially identical to the previous proposition's. If Alice has more than $b+1$ chips, she just bets 
$b+1$ and wins. If she has $b+1$ chips, then unless the congruence condition holds, she can win by betting $b+1$ chips if $b$ is less than $2^{n-1}-1$ and $b$ chips if $b$ is greater than $2^{n-1}-1$. If $b$ is equal to $2^{n-1}-1$ then Bob can bet all his chips to return with either $b^*$ chips or $b+1$ chips, thus forcing a draw.
\end{proof}

Now $\Ult^n$ is clearly stable, since there is only one move available to each player from each position, so the above cases can be used to determine the outcome of $\Ult^n$ for all possible chip counts using the periodicity results of Section~\ref{periodicity section}.

\begin{corollary}
Alice wins $\Ult^n(a^*, b)$ if and only if either $a$ is greater than $b$ or $a$ and $b$ are equal but not congruent to $2^{n-1}-1 \ (\mathrm{mod} \ 2^n)$. Similarly, Alice wins $\Ult^n(a, b^*)$ if and only if either $a$ is greater than $b+1$, or $a$ is equal to $b+1$ and not congruent to $ 2^{n-1}\ (\mathrm{mod} \ 2^n)$.
\end{corollary}

Using these computations for the Ultimatum games, we can construct another interesting example: a discrete bidding game $G$ whose critical threshold $R(G)$ is rational, but when Alice's proportion of the total number of chips is exactly equal to $R(G)$ the outcome is not periodic in the total number of chips.  For finite games, such behavior is impossible by the periodicity results of Section~\ref{periodicity section}.

\begin{proposition}
Let $G$ be the game where on the first move, either player can choose one of $\Ult^1, \Ult^3, \Ult^5, \ldots$, and then the 
players play the chosen game with the current chip counts. Then $R(G)=1/2$, but the sequence of results of $G(k^*, k)$ 
is aperiodic.
\end{proposition}

\begin{proof}
Since the game is symmetric, it cannot be a win for Bob. If Alice either bets 1 or uses the tiebreaking advantage in a 
0-0 tie, then she will be playing an ultimatum game with fewer chips and cannot win. So her only chance at forcing a 
win is to let Bob win for 0 chips. So Bob picks any ultimatum game. If he can pick a game for which $\Ult^n(k^*, k)$ is a 
draw, $G(k^*, k)$ is a draw; otherwise, $G(k^*, k)$ is an Alice win. 

However, the values of $k$ for which $\Ult^n(k^*, k)$ is a draw are those for which $k\equiv 2^{n-1}-1$ (mod $2^n$). Only games $\Ult^n$ for odd $n$ are in play, so Bob can draw whenever $k$ is congruent to 0 mod 2, or 3 mod 8, or 15 mod 32, and so on. It is easily seen that the set of $k$ which satisfy these congruence conditions is not a finite union of arithmetic progressions and hence the outcome is not periodic in $k$.
\end{proof}

\section{A partial order on games} \label{partial order section}

If we know the optimal moves for Alice and Bob, or if both the game $G$ and the total number of chips $k$ are fixed, then the critical threshold $f(G,k)$ is easy to compute using the formulas in Section~\ref{discrete section}.  On the other hand, for fixed $G$, the tree of optimal moves may vary as $k$ varies.  In this section, we discuss situations in which one can find moves that are optimal for all $k$.

We define a partial order on games by setting $G \leq G'$ if $f(G,k) \leq f(G',k)$ for all $k$.  In other words, $G \leq G'$ means that, regardless of the total number of chips, if Alice has enough chips to win $G'$ then she also has enough chips to win $G$.  Then $G$ and $G'$ are equivalent in this partial order if $f(G,k) = f(G',k)$ for all $k$.  The point of this partial order is that if $G \leq G'$, then Alice will always prefer to move to a position $v$ such that $G_v$ is equivalent to $G$ rather than a position $v'$ such that $G_{v'}$ is equivalent to $G'$, regardless of the total number of chips.

\begin{example} Let $\A$ be the equivalence class of games in which Alice always wins, and let $\B$ be the equivalence class of games in which Bob always wins.  Then
\[
\A \leq G \leq \B
\]
for every game $G$.
\end{example}

If $G$ and $G'$ are games with many legal moves available from each position, then it may be very difficult to tell whether $G$ and $G'$ are comparable in this partial order.  In practice, we can sometimes compare relatively complicated games $G$ and $G'$ by relating them to games with fewer legal moves.

\begin{example}
For a positive integer $n$, let $A^n$ be the game in which Alice wins if she makes any of the first $n$ moves, and Bob wins otherwise.  Roughly speaking, we think of $A^n$ as a sudden-death game in which Alice has an advantage of order $n$.  Similarly, let $B^n$ be the game in which Bob wins if he makes any of the first $n$ moves and Alice wins otherwise.  We write $\E$ for the equivalence class of games in which the player with more chips always wins.  One game in the class of $\E$ is the game in which the first player to move wins.  Then $A^1$ and $B^1$ are both equivalent to $\E$, and 
\[
\A < \cdots < A^3 < A^2 < \E < B^2 < B^3 < \cdots < \B.
\]
The games $A^n$ and $B^n$ are useful for comparing positions, since $A^n \leq G$ for any game $G$ in which Bob wins if he makes all of the first $n$ moves, and $G' \leq A^n$ for any game $G$ in which Alice wins if she makes any of the first $n$ moves, and similarly for $B^n$.
\end{example}

\begin{example} \label{symmetric}
Games that are equivalent to $\E$ in this partial order are not necessarily obviously symmetric.  For instance, Tic-Tac-Toe played starting from any of the following six positions is equivalent to $\E$.
\begin{center}
\includegraphics{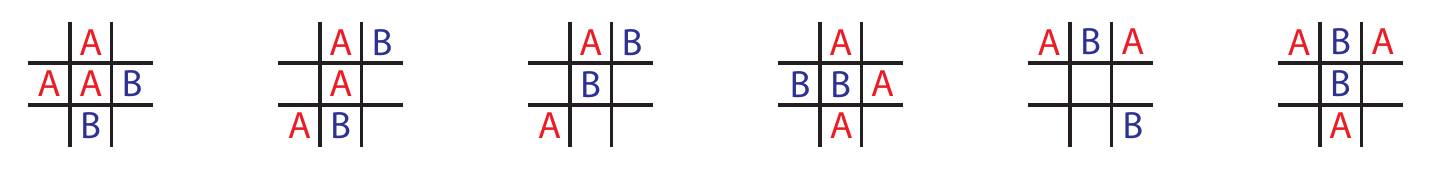}
\end{center}
From each of these positions, Alice wins if she makes two moves before Bob does, and Bob can prevent Alice from winning if he makes two moves before Alice does.  It follows that each of these positions is equivalent to $A^2 \wedge B^2$, which is equivalent to $\E$.  
\end{example}

Recall that the wedge sum $G \wedge H$ of two games $G$ and $H$ is the game where the first player to move can choose between the starting position for $G$ and the starting position for $H$.  Although $G \wedge H$ is the same as $H \wedge G$, when one game is clearly better for each player we write the game that Alice prefers on the left.  For example, $A^2 \wedge B^3$ is a game in which Alice can move to a position equivalent to $A^2$ and Bob can move to a position equivalent to $B^3$.  This notation is convenient for working with equivalence classes of games and our partial order, since $G \wedge H$ is equivalent to $G' \wedge H$ if $G$ is equivalent to $G'$, and $G \wedge H \leq G' \wedge H'$ if $G \leq G'$ and $H \leq H'$.  Simple relations between games are also easily expressible in this notation.  For instance, we have equivalences of games
\[
A^n \equiv \A \wedge A^{n-1} \mbox{ \ \ and \ \ } \E \equiv G \wedge \overline G,
\]
for any integer $n \geq 2$ and any game $G$.

In Section \ref{TTT section}, we completely analyze bidding Tic-Tac-Toe by comparing positions in Tic-Tac-Toe with iterated wedge sums of games $\A$, $\B$, $A^m$ and $B^n$.  For instance, we show that Tic-Tac-Toe played from the position
\begin{center}
\includegraphics{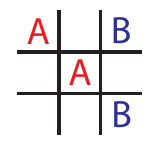}
\end{center}
is equivalent to $(\A \wedge B^2) \wedge \B$.

\section{Bidding Tic-Tac-Toe} \label{TTT section}

In this section, we use comparisons with bivalent games and the partial orders defined in Section~\ref{partial order section} to completely analyze Tic-Tac-Toe for both real-valued bidding and discrete bidding with arbitrary chip counts.

Since Tic-Tac-Toe is symmetric, to determine whether or not Alice wins Tic-Tac-Toe$(a,b)$ for all possible chip counts
\[
(a,b) \in (\N \times \N^*) \cup (\N^* \times \N).
\]
Indeed, Bob wins Tic-Tac-Toe$(a,b)$ if and only if Alice wins Tic-Tac-Toe$(b,a)$, and the outcome is a tie if and only if neither Alice nor Bob wins.  Therefore, we will focus our analysis on the game $\T$ that is played just like Tic-Tac-Toe except that Bob is declared the winner of any game that would normally be a tie.  In particular, Alice wins Tic-Tac-Toe if and only if she wins $\T$, so these games have the same critical thresholds.

\subsection{Optimal moves}  We begin by determining the optimal moves in Tic-Tac-Toe, as much as possible, using the partial order defined in Section~\ref{partial order section}.  This simple approach determines all of the optimal moves except Alice's optimal move from an empty board, which we reduce to two possibilities.  In Section~\ref{TTT tables} we compute the number of chips that Alice needs to win for every possible chip total for each of these two possibilities, using the recursion from Theorem~\ref{discrete recursion} for small chip totals and the periodicity results in Theorems~\ref{periodicity for Alice} and \ref{periodicity for Bob} for higher chip totals.

\begin{theorem} \label{stable TTT}
The tree of moves for $\T$ shown in Figure 1 is optimal for real-valued bidding.  These moves are also optimal for discrete bidding if the total number of chips is not equal to five.  Furthermore, the critical thresholds of each position reached through these optimal moves is as indicated in the figure.  In particular,
\[
R(\T) = 133/256.
\]
For discrete bidding with five chips, a tree of optimal moves is shown in Figure 2, below.
\end{theorem}

\noindent The critical threshold $R(\T) = 133/256$ was computed independently by Theodore Hwa \cite{Hwa06}. 

\vspace{-10pt}

\begin{center}
\includegraphics[width=6.25in]{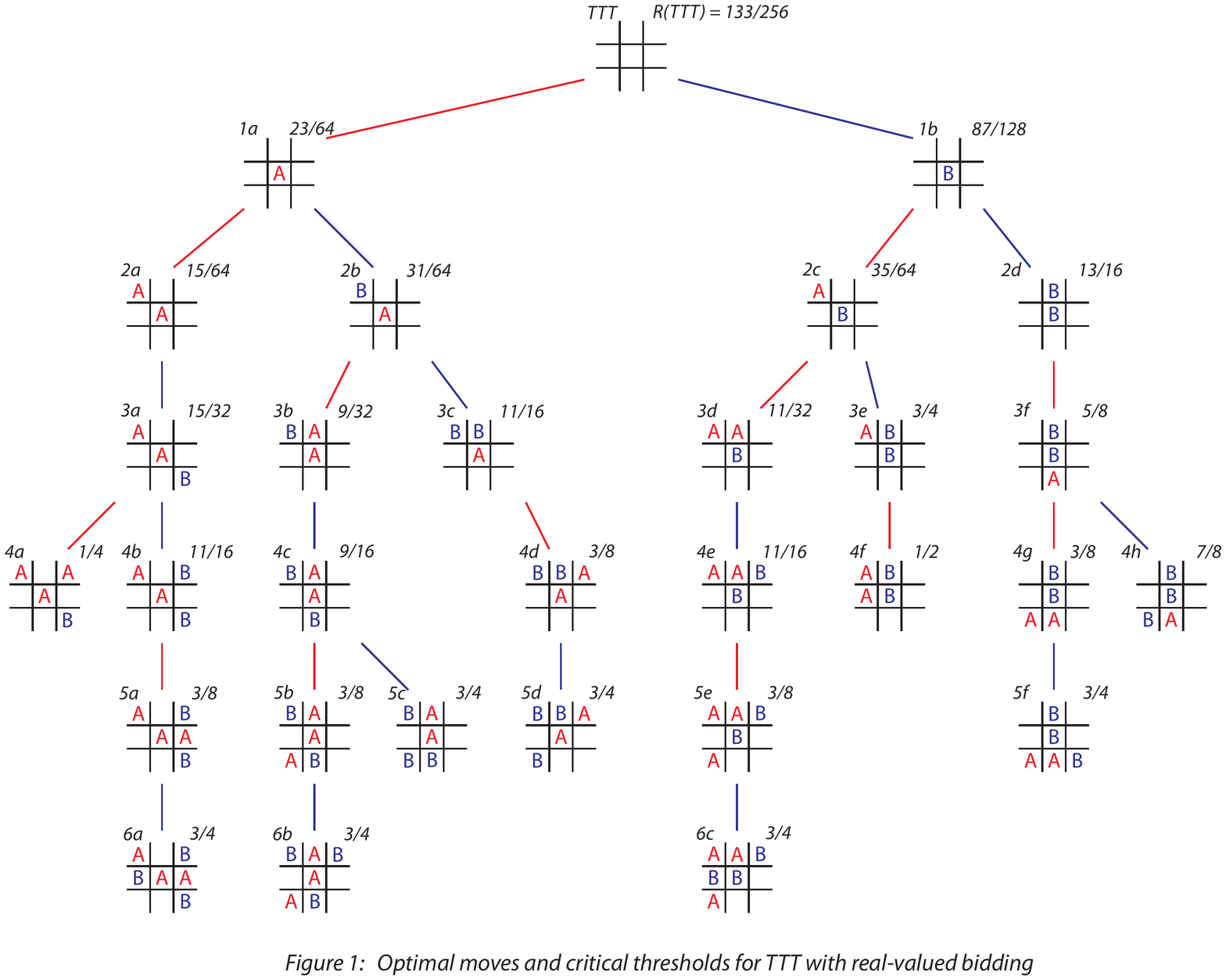}
\end{center}

Our general approach to studying $\T$ is to work our way backward from end positions.  Knowing which sequence of moves is optimal late in the game greatly simplifies the analysis of earlier positions, by reducing to comparisons with wedge sums of games equivalent to $\A$, $\B$, $A^n$ or $B^n$.

For discussing $\T$ and other games, it is useful to have some basic language for describing positions, although we attempt to keep jargon to a minimum.  We say that Alice has a \emph{threat} if she can win on the next move, and Bob has a \emph{threat} if he can prevent Alice from winning with his next move.  If Alice has a threat from a position $v$ then, in the partial order on games discussed in Section \ref{partial order section},
\[
G_v \leq \E,
\]
and if Bob has a threat from $v'$ then $\E \leq G_{v'}$, where $\E$ is the game in which the first player to move wins.  We say that Alice has a \emph{double threat} if she can win if she gets either of the next two moves, and Bob has a \emph{double threat} or \emph{triple threat} if he can prevent Alice from winning if he gets any of the next two moves or three moves, respectively. 

\begin{example}
Bob has a triple threat from the position
\begin{center}
\includegraphics{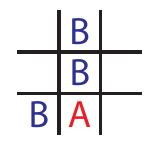}
\end{center}
since he can prevent Alice from winning if he moves anywhere in the right column.  This position is equivalent to $B^3$, since Alice wins if she gets three moves in a row.  On the other hand, Alice has a double threat from the position
\begin{center}
\includegraphics{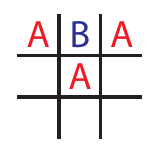}
\end{center}
but this position is strictly better for Alice than $A^2$, since Bob cannot prevent Alice from winning if he gets two moves in a row.
\end{example}

We say that a move is a \emph{block} if it goes from a position in which the opponent has a threat to a position in which the opponent does not have a threat.  A move is a \emph{counterattack} if it is a block that also creates a threat.

\begin{lemma} \label{counterattack}
If a counterattack is available, then the optimal move must be a block.
\end{lemma}

\begin{proof}
A counterattack moves to a position that is at least as good as $\E$, and any move that is not a block moves to a position that is no better than $\E$.
\end{proof}

\noindent  In Tic-Tac-Toe, if a player has a block available then it is always unique.  Therefore, counterattacks are always optimal in Tic-Tac-Toe.  Note however, that blocks are not always optimal in Tic-Tac-Toe.  See the analysis of Alice's move from position 3e, below, for an example where blocking is not optimal.

\begin{proposition}
Every move in Figure 1 other than Alice's first move is optimal for all chip counts.
\end{proposition}

\begin{proof}

There are twenty-nine moves in the stable game tree above, so we must analyze each of these twenty-nine moves.  The moves from positions 5a, 5b, 5e, 4a, 4d, 4e, 4f, 3d, and 3e are counterattacks, and by Lemma \ref{counterattack}, counterattacks are always optimal.  We now analyze each of the remaining moves other than Alice's first move, moving from bottom to top and from left to right in Figure 1.

 \vspace{10pt} 

\noindent  \emph{Alice's move from 4c:}  Wherever Alice moves, Bob can create a double threat by moving top right or bottom left.  Therefore, the best game Alice can hope to reach is $\A \wedge B^2$, which she reaches by moving bottom left.

 \vspace{10pt} 

\noindent \emph{Bob's move from 4c:}  Wherever Bob moves, Alice can win if she gets two moves in a row.  Therefore, the best game Bob can hope to move to is $B^2$, which he reaches by moving bottom left.

 \vspace{10pt} 

\noindent  \emph{Alice's move from 3a:}  Wherever Alice moves, Bob can win if he gets two moves in a row by filling either the right column or the bottom row.  Therefore, the best game Alice can hope to move to is $A^2$, which she reaches by moving top right.

 \vspace{10pt} 

\noindent \emph{Bob's move from 3a:}  Bob's move in either corner creates $(\A\wedge B^2) \wedge \B$, which is better for him than $A^2 \wedge \B$ and hence better than any position from which Alice can produce a double threat.  If Bob moves on a side, then Alice can create a double threat by moving in a corner.  Therefore, Bob's corner move is optimal.

 \vspace{10pt} 

\noindent \emph{Bob's move from 3b:}  If Bob does not block, then the best game he can reach is $\E$, which is equivalent to $A^2 \wedge B^2$.  Blocking gives $(\A\wedge B^2) \wedge B^2$.  Since $\A\wedge B^2$ is strictly better than $A^2$, it follows that blocking is optimal for Bob.

 \vspace{10pt} 

\noindent \emph{Alice's move from 3e:}  This is an interesting move, since Alice's optimal strategy is not to block, but instead to create a threat of her own to reach a game equivalent to $\E$.    If Alice blocks, then the resulting position is $A^2 \wedge B^3$, since Bob can move lower left to create a triple threat, and this is strictly worse for Alice than $\E$.   Of course, if Alice moves in the right column, then Bob has a threat but Alice does not, which is again worse than $\E$.  Therefore, moving in the left column is optimal for Alice.

\vspace{10 pt}

\noindent  \emph{Alice's move from 3f:}  Wherever Alice moves, Bob can produce a double threat by moving in a bottom corner.  Therefore Alice's bottom corner move, which gives $\A\wedge B^2$, is optimal.

\vspace{10 pt}

\noindent \emph{Bob's move from 3f:}  Wherever Bob moves, Alice can win if she gets the next three moves in a row.  Therefore, Bob's move in a bottom corner, which produces a triple threat, is optimal.

\vspace{10 pt}

\noindent \emph{Bob's move from 2a:}  If Bob does not block, then he must get the next move as well.  And if blocking is not optimal, then his optimal next move must also be nonblocking (since otherwise he may as well have blocked the first time).  If Bob makes two moves without blocking, then the best he can do is create a symmetric threat, reaching $\E$.  However, if he does block, then in two moves he can reach a position in which he has a threat and Alice does not, which is strictly better than $\E$.  Therefore, blocking is optimal.

\vspace{10 pt}

\noindent \emph{Alice's move from 2b:}  Suppose Alice  moves upper right or lower left.  Then Bob's next move must be a counterattack, and Alice's move after that must be a counterattack as well, and the resulting position after five moves is the one labeled 5a in the game tree.  Therefore either of these moves leads to $\A\wedge (TTT_{5a} \wedge \B)$.

Suppose Alice moves lower right.  Then Bob can move upper right forcing Alice to counterattack in the top center, leading to position 5a.  Since this move does not create a threat, it is strictly worse than $\A\wedge (TTT_{5a} \wedge \B)$.

Finally, suppose Alice moves top or left.  Then we have seen in the analysis of position 4c that Bob's optimal response is to block.  It follows that moving top or left leads to $\A\wedge (TTT_{5b} \wedge B^2)$.  Since $TTT_{5a}$ and $TTT_{5b}$ are both equivalent to $\A\wedge B^2$, it follows that any move other than lower right is optimal.  In particular, moving top, as shown in the game tree, is optimal.

\vspace{10 pt}

\noindent \emph{Bob's move from 2b:}  If Bob moves top or left, then the resulting position is equivalent to $(\A\wedge B^2) \wedge \B$.

Suppose Bob moves upper right or lower left.  Then Alice must counterattack.  No matter what Bob does next, Alice will still be able to win if she gets the next two moves.  Therefore, this is no better for Bob than $(\A\wedge B^2) \wedge \B$. 

Finally, suppose Bob moves lower right.  Then Alice can move in the top, forcing Bob to counterattack, and again Alice can win if she gets the next two moves.  So this is still no better than $(\A\wedge B^2) \wedge \B$.  Therefore, moving top or left is optimal.

\vspace{10 pt}

\noindent \emph{Alice's move from 2c:}  If Alice moves top or left, a series of counterattacks ensues, as shown in the stable game tree.  This position gives $\A\wedge \big( (\A\wedge B^2) \wedge \B \big)$.

If Alice moves upper right or lower left, then Bob counterattacks and then best Alice can do in response is to reach a symmetric position.   Therefore, these moves give $\A \wedge ( \E \wedge \B)$ which is strictly worse for Alice than moving top or left.

Finally, suppose Alice moves lower right.  Then Bob can move on a side and force a series of counterattacks.  After such a move from Bob, this position is then $(\A\wedge B^2) \wedge \B$, which is the same as if Alice moves top or left and then Bob counters.  However, since lower right does not create a threat for Alice, this is strictly worse than moving top or left, and it follows that top and left are Alice's only optimal moves.

\vspace{10 pt}

\noindent \emph{Bob's move from 2c:}  Wherever Bob moves, Alice can create a threat on the next move.  Therefore, the best game Bob can hope to reach is $\E \wedge \B$, which he gets to by moving top or left.

\vspace{10 pt}

\noindent \emph{Alice's move from 2d:}  If Alice does not block, then she must get two moves in a row, and if she gets two moves in a row then the best game she can reach without blocking is $\E$.  Since she can reach position 4g in two moves, which is strictly better than $\E$, it follows that blocking is her only optimal move.

\vspace{10 pt}

\noindent \emph{Alice's move from 1a:}  If Alice moves in a corner, then the resulting game is $\A\wedge (A^2 \wedge TTT_{3a})$.  We must show that Alice cannot do better by moving on the top.  

Suppose Alice moves on the top.  Then Bob can block on the bottom.  If Bob gets the next move, then he can force a series of counterattacks, guaranteeing that this is no worse for him than 3a.  Therefore, it will suffice to show that if Alice gets the next move after Bob blocks, then the best she can reach is a game equivalent to $A^2$.  If Alice moves in the top row, then she has a double threat and Bob can win if he gets two moves in a row, so this is exactly $A^2$.  If Alice moves in the middle row, then Bob can block, reaching a position
\begin{center}
\includegraphics{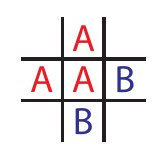}
\end{center}
that is symmetric by Example \ref{symmetric}.  And if Alice moves in the bottom row, then Bob can block, reaching a position
\begin{center}
\includegraphics{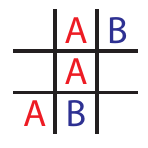}
\end{center}
that is symmetric by Example \ref{symmetric}.  It follows that Alice cannot do better than moving in a corner.

\vspace{10 pt}  

\noindent \emph{Bob's move from 1a:}  If Bob moves in a corner, then the resulting game is $(\A \wedge TTT_{4c}) \wedge (TTT_{4d} \wedge \B)$.  We must show that Bob cannot do better by moving on top.

Suppose Bob moves on top and Alice wins the next move.  Then Alice can move on the left side, and we claim that the resulting position is equivalent to $\A \wedge \E$.  If Bob does not block, then the best he can reach is a symmetric position.  If Bob does block, then the resulting position is still symmetric, since Alice's best move is bottom left, which gives $A^2$, and Bob's best move is top right, which gives $B^2$.  Now $\A \wedge \E$ is strictly worse for Bob than $\A \wedge TTT_{4c}$, which results if Bob moves in a corner and Alice wins the next move.

Suppose Bob moves on top and gets the next move as well.  The only way he can do better than starting in a corner is by moving on a side again.  Suppose he moves left on his second move.  Then Alice can move in the upper left, and Bob's counterattack leads to a position $B^2$.  If Bob moves on the bottom on his second move, then Alice again moves in the upper left, and Bob's counterattack is worse for him than $B^2$.  It follows that after making two consecutive moves on sides, Bob's position is strictly worse than the position $(\A \wedge B^2) \wedge \B$ that he would reach by moving first in a corner and then on an adjacent side.  Therefore, the corner move from position 1a is optimal for Bob.

\vspace{10 pt} 

\noindent \emph{Alice's move from 1b:}  Suppose Alice moves on the top instead.  If Alice gets the following move as well, she could only do better by moving on another side, either left or bottom.  If she plays on the left, then Bob can move upper left, and Alice will need to block in order to win.  Therefore, Alice blocks and Bob's next move creates $B^2$.  In other words, after Alice moves top and left, the position is strictly worse for Alice than $\A \wedge \big( (\A \wedge B^2) \wedge \B \big)$.  On the other hand, if Alice moves top and then bottom, then Bob can move left and the best Alice can do is to reach the position
\begin{center}
\includegraphics{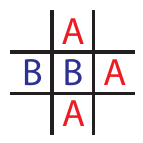}
\end{center}
that is symmetric by Example \ref{symmetric}.  Therefore, if Alice moves top and then bottom, the resulting position is strictly worse for Alice than $\A \wedge (\E \wedge \B)$.  Both of these situations are worse than the position 3d, which is equivalent to $\A \wedge \big( (\A \wedge B^2)  \wedge \B \big)$, which Alice gets to by playing top left and then top.

Next, suppose Alice moves on top and Bob makes the next move.  Bob can play top right, and then Alice needs to take the lower left corner to have any chance of winning.  So Alice blocks, and reaches a position
\begin{center}
\includegraphics{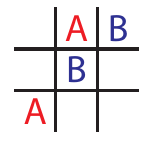}
\end{center}
that is symmetric, by Example \ref{symmetric}.  In particular, if Alice moves top and then Bob makes the next move, the result is no better for Alice than $\E \wedge \B$.  It follows that Alice cannot do better than reaching the position $\T_{4d} \wedge (\E \wedge \B)$ that she gets by playing in the corner.

\vspace{10 pt}

\noindent \emph{Bob's move from 1b:}  Suppose Bob moves in the upper left corner instead.  Then Alice can block in the lower right.  If Bob gets the next move, he can do no better than to reach $B^2$, while Alice can move on the bottom to reach a position $\A \wedge B^2$.  Therefore, if Bob moves in a corner, his position is no better than $\big( (\A \wedge B^2) \wedge B^2 \big) \wedge \B$, which is worse for him than $\T_{2d}$, which is equivalent to $\big((\A \wedge B^2) \wedge B^3 \big) \wedge \B$.

\vspace{10pt}

\noindent \emph{Bob's first move:}  Suppose Bob makes the first move, but not in the center.  If Alice makes the second move, she can go in the center.  By our analysis of Bob's move from 1a, 
\begin{center}
\includegraphics{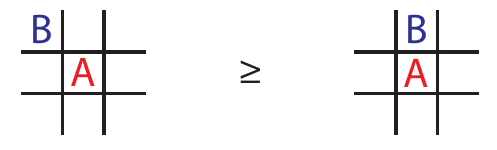}
\end{center}
and by inspection of the tree of optimal moves from positions 2b and 2c in Figure 1, it is straightforward to check that $TTT_{2c} \geq TTT_{2b}$.  Therefore, if Alice makes the second move, Bob would have done best to move in the center.

Suppose Bob makes the second move as well.  If Bob's optimal first move is not in the center, then his optimal second move must be not in the center as well.  Suppose Bob starts with any of the following two moves:
\begin{center}
\includegraphics{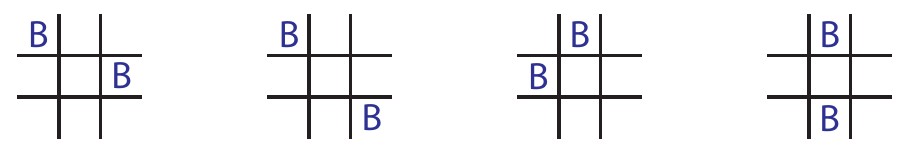}
\end{center}
Then Alice can respond by moving in the center.  After this, Bob can do no better than to create a double threat on his next move.  Therefore each of the four positions above is no better than $\big( (\A \wedge B^2) \wedge B^2 \big) \wedge \B$.

Similarly, if Bob starts with
\begin{center}
\includegraphics{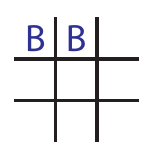}
\end{center}
and then Alice blocks, the resulting position is still no better for Bob than $\big( \A \wedge B^2) \wedge B^2 \big) \wedge \B$.  

Finally, if Bob starts with
\begin{center}
\includegraphics{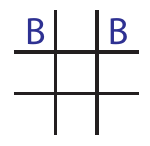}
\end{center}
then Alice can block.  On the next move, Bob can do no better than to move in the center and create a $B^3$, while Alice can move in the center to create a position that is worse for Bob than $\A \wedge B^2$, since Alice can win on the next move, but Bob cannot create a double threat.  In particular, all of these scenarios are worse for Bob than $\big( \A \wedge B^2) \wedge B^3 \big) \wedge \B$, which is equivalent to $TTT_{2d}$.  It follows that Bob's only optimal first move is in the center.
\end{proof}  

\begin{lemma} \label{Alice's first two moves}
If center is not an optimal first move for Alice then Alice's optimal first two moves are in adjacent corners.
\end{lemma}

\begin{proof}
Suppose Alice makes an optimal first move, but not in the center.  If Bob moves second then he can move in the center.  However, by our analysis of Alice's move from position 1b, and by inspection of the trees of optimal moves from positions 2b and 2c, we have
\begin{center}
\includegraphics{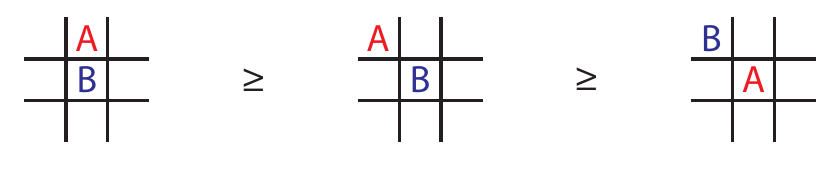}
\end{center}
Therefore, Alice must reach a position better than 2a by making the first two moves, neither in the center.
We consider all possible positions that Alice can reach after making the first two moves.  If Alice's first two moves are any of the following
\begin{center}
\includegraphics{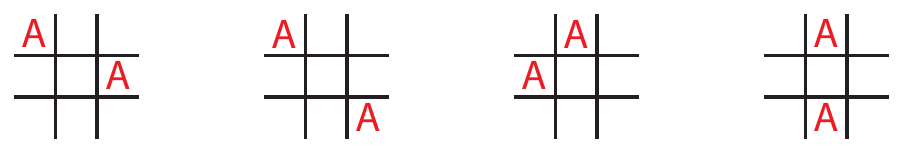}
\end{center}
then Bob moves in the center.  Then Alice cannot reach a position better than $A^2$ on her next move, while Bob can move in either the top left corner or the top side to reach a position that is worse for Alice than 4b $\equiv (\A \wedge B^2) \wedge \B$.  Similarly, if Alice makes her first two moves
\begin{center}
\includegraphics{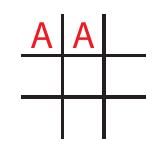}
\end{center}
then Bob can block.  After Bob blocks, if Alice moves next she can get no better than $A^2$, while Bob can move in the center to reach position 4e, which is equivalent to 4b.  In particular, none of these scenarios is better for Alice than the positions that she would reach by starting with 2a, and the only remaining possibility is that Alice could make her first two moves in adjacent corners.  This proves the lemma.
\end{proof}

\begin{lemma}
Suppose center is not an optimal first move for Alice, for some fixed total number of chips.  Then every move in Figure 2 is optimal.
\end{lemma}

\vspace{-10pt}

\begin{center}
\hspace{-5pt}\includegraphics[width=4.5in]{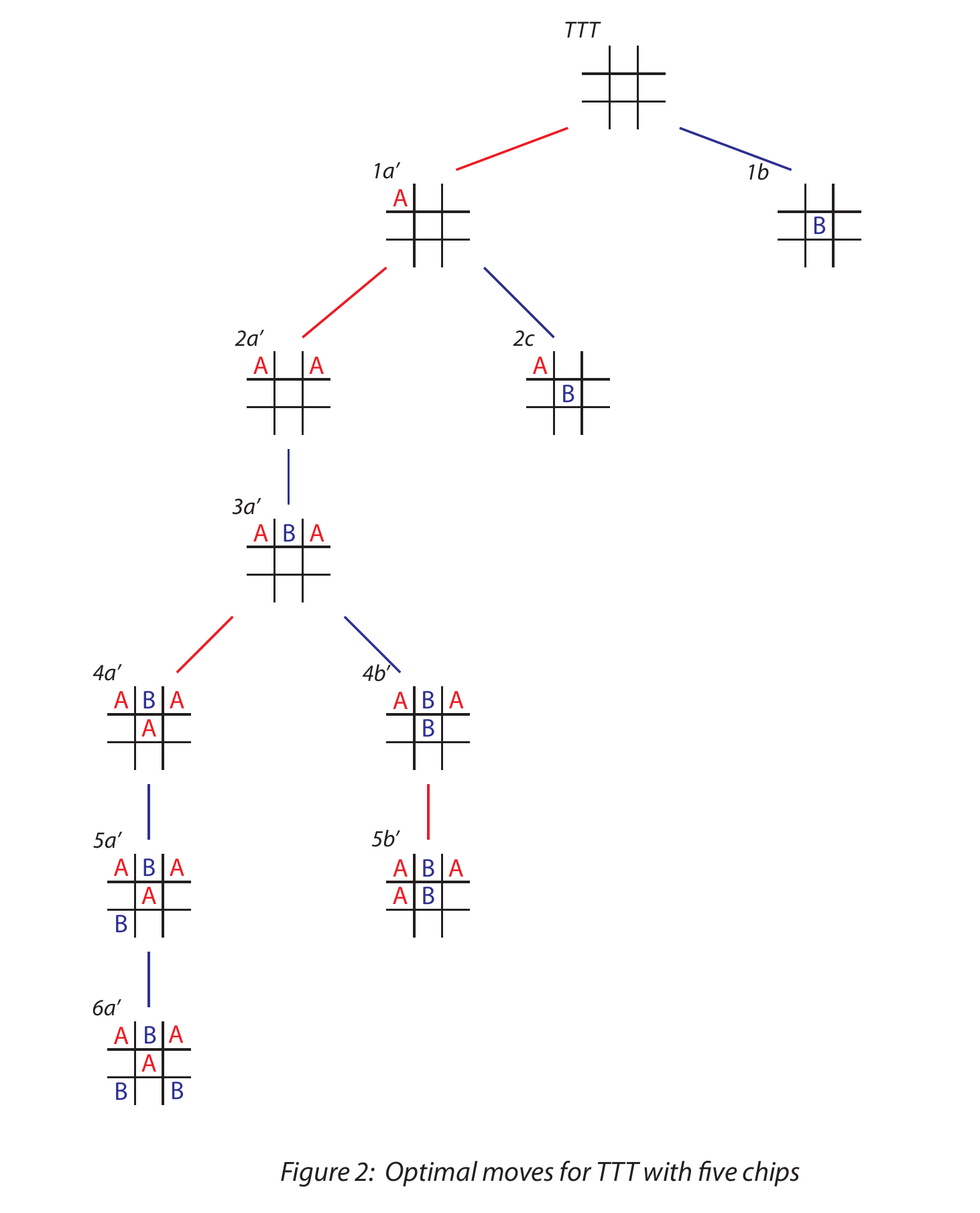}
\end{center}

\begin{proof}
Bob's move from 5a$'$ is a counterattack, Alice's first move and her move from 1a$'$ are optimal by Lemma \ref{Alice's first two moves}, and Bob's first move is optimal by Theorem \ref{stable TTT}.  We analyze the remaining six moves as follows.

\vspace{10 pt}

\noindent \emph{Bob's move from 4a$'$:}  If Bob does not block one of Alice's threats then he must win the following two moves as well.  However, if he gets three moves in a row then he can also win by blocking with the first one.

\vspace{10 pt}

\noindent \emph{Alice's move from 4b$'$:}  Wherever Alice moves, the resulting position is equivalent to $\E$.

\vspace{10 pt}

\noindent \emph{Alice's move from 3a$'$:}  Alice's move in the center creates a position equivalent to $\A \wedge (\A \wedge B^2)$.  If Alice plays anywhere else, then Bob moves in the center to create a position that is no better for Alice than $\E$, and hence worse than $\A \wedge B^2$.  Indeed, if Alice does not move in the center column, then Bob's move in the center creates a threat, and if Alice moves in the bottom then Bob's move in the center creates a position
\begin{center}
\includegraphics{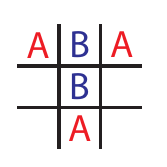}
\end{center}
that is equivalent to $\E$ by Example \ref{symmetric}.

\vspace{10 pt}

\noindent \emph{Bob's move from 3a$'$:}  Moving center creates a position equivalent to $B^2$, or $\E \wedge \B$.  If Bob moves anywhere else, then Alice can move center to reach a position in which she has a threat and Bob does not, which is worse for Bob than $\E$.

\vspace{10 pt}

\noindent \emph{Bob's move from 2a$'$:}  If Bob does not block then he must get the next move as well.  If he blocks on the next move, he may as well have blocked the first time, and if he doesn't block on his second move, then he has to get a third move in a row.  With three moves in a row, he can win after blocking with his first move.

\vspace{10 pt}

\noindent \emph{Bob's move from 1a$'$:}  If Bob plays anywhere other than center or bottom right then Alice can play in the center to reach a position that is no better for Bob than 3d,  Meanwhile, if Bob gets two moves in a row he cannot prevent Alice from winning if she gets the next two moves.  Therefore center is better than any other option except possible bottom right. 

Suppose Bob plays bottom right.  Then Alice can play top right and the best Bob can do is to block to reach a position
\begin{center}
\includegraphics{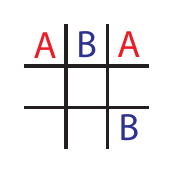}
\end{center}
that is equivalent to $\E$ (Example \ref{symmetric}).  Again this is worse for Bob than if he had played center.
\end{proof}

\subsection{Chip tables}  \label{TTT tables}

Having determined all of the optimal moves for bidding Tic-Tac-Toe, except for Alice's first move which we have reduced to two possibilities, we now compute the total number of chips that Alice needs to win from each position that can be reached by these optimal moves.  These computations are straightfoward applications of the recursion given by Theorem~\ref{discrete recursion} and the periodicity results in Theorem~\ref{periodicity for Alice} and \ref{periodicity for Bob}.  As a consequence of these computations, we find that Alice's optimal first move depends on the total number of chips, as follows.

\begin{theorem} \label{final TTT}
Center is an optimal first move for Alice if and only if the total number of chips is not five.  Corner is an optimal first move for Alice if and only if the total number of chips is 0, 1, 2, 3, 4, 5, 6, 7, 9, 11, 12, 13, 14, 19, 20, 22, or 26.
\end{theorem}

\begin{proof}
We prove the theorem by explicitly computing the number of chips that Alice needs to win assuming a first move in the center or a first move in the corner, for all possible chip totals.  Let $f'(TTT,k)$ be the number of chips that Alice needs to win, assuming a first move in the center, and $f''(TTT,k)$ be the number of chips that Alice needs to win, assuming she makes the first move in the corner, if she wins the first bid.

\begin{center}
\small
\vskip .5cm \begin{tabular}{|r||*{12}{l|}}\hline 
$256n+$&+0&+1&+2&+3&+4&+5&+6&+7&+8&+9&+10&+11\\ \hline \hline 
0+&1&1&1*&2*&2*&3*&4&4&4*&5*&5*&6*\\ \hline
12+&7&7&8&8*&9&9&10&10*&11&11*&12*&12*\\ \hline
24+&13&13&14&14*&15&15*&16&17&17*&17*&18&18*\\ \hline
36+&19*&20&20*&20*&21*&21*&22&23&23*&24&24*&25\\ \hline 
48+&25*&26&26&27&27*&28*&28*&29&29*&30&30*&31*\\ \hline
60+&31*&32&33&33&34&34*&34*&35*&36&36*&37&37*\\ \hline
72+&37*&38*&39&40&40&40*&41&41*&42&42*&43&44\\ \hline
84+&44*&44*&45*&45*&46&46*&47*&48&48*&48*&49&50\\ \hline 
96+&50*&51&51*&52&53&53&53*&53*&54*&55&55*&56\\ \hline
108+&57&57&57*&58&59&59&59*&60*&60*&61*&62&62\\ \hline
120+&62*&63*&63*&64*&65&65&66&66*&67*&67*&68&69\\ \hline
132+&69&70&70*&70*&71&72&72&73&73*&73*&74*&75\\ \hline 
144+&75*&75*&76*&77&77*&78&79&79&79*&79*&80*&81\\ \hline
156+&81*&82&82*&83*&84&84&84*&85&86&86*&87&87\\ \hline
168+&88&88&88*&89*&90&90*&91&91*&92&92*&92*&93*\\ \hline
180+&94&95&95&95*&96&96*&97&98&98&98*&99*&99*\\ \hline 
192+&100*&101&101&102&102*&103&103*&104&104&105&105*&106*\\ \hline
204+&106*&107&107*&108&108*&109&109*&110*&111&111&112&112\\ \hline
216+&112*&113&114&114*&115&115&115*&116*&117&117*&118&118*\\ \hline
228+&119*&119*&120&120&121&121*&122&122*&123*&123*&124&124*\\  \hline 
240+&125*&125*&126&127&127&128&128*&128*&129&130&130&131\\ \hline
252+&131*&131*&132*&133&&&&&&&&\\ \hline 
\end{tabular}
\normalsize
\vskip .5cm
\emph{Figure 3: Critical thresholds if Alice moves in the center;  $f'(TTT,256n + \ ) = 133n + \ \ $.}
\end{center}

\noindent The entries in the interior of the preceding table are critical thresholds, while the position of the entry determines the total number of chips.  For instance, the last entry in the second row is 12*, which means that the critical threshold $f'(TTT, 256n+23)$ is $133n + 12^*$, for any nonnegative integer $n$.

\begin{center}
\small
\begin{tabular}{|r||*{12}{l|}}\hline 
$256n+$&+0&+1&+2&+3&+4&+5&+6&+7&+8&+9&+10&+11\\ \hline \hline 
0+&1&1&1*&2*&2*&3&4&4&5&5*&6&6*\\ \hline
12+&7&7&8&9&9*&9*&10*&10*&11&12&12*&13\\ \hline
24+&13*&13*&14&15&15*&16&16*&17*&18&18&18*&19*\\ \hline
36+&20&20*&21&21*&22&22*&23&23*&24&24*&25&26\\ \hline 
48+&26*&27&27&27*&28*&29&29*&30&31&31&31*&32*\\ \hline
60+&32*&33&34&34&35&35*&35*&36*&37&37&38&38*\\ \hline
72+&39&39*&40*&41&41&41*&42&43&43*&44&44*&45\\ \hline
84+&45*&46&46*&47&47*&48&48*&49*&50&50&50*&51*\\ \hline 
96+&52&52*&53&54&54*&54*&55&55*&56&57&57*&57*\\ \hline
108+&58*&58*&59&60&61&61&61*&62&62*&63&64&64\\ \hline
120+&65&65*&65*&66*&67&67&68&68*&69*&69*&70&71\\ \hline
132+&71&71*&72*&72*&73*&74&74*&75&75*&75*&76*&77*\\ \hline 
144+&78&78&79&79&79*&80*&81&81*&82&82&82*&84*\\ \hline
156+&84&84*&85&86&86*&86*&87&88&88*&89&89*&90\\ \hline
168+&90*&91&91*&92&92*&93&93*&94*&95&95*&95*&96\\ \hline
180+&97&97*&98&98*&99*&99*&100&101&101&101*&102*&102*\\ \hline 
192+&103*&104&104&105&105*&105*&106*&107&107*&108&109&109*\\ \hline
204+&109*&110&110*&111*&112&112*&113&113*&114&114*&115&115*\\ \hline
216+&116&116*&117&118&118*&118*&119&120&120*&121&121*&122*\\ \hline
228+&123&123&123*&124&124*&125*&126&126&127&127&127*&128*\\  \hline 
240+&129*&129*&130&130*&131&131*&132*&132*&133*&134&134&135\\ \hline
252+&135*&135*&136*&137&&&&&&&&\\ \hline 
\end{tabular}
\normalsize
\vskip .5cm
\emph{Figure 4: Critical thresholds if Alice moves in the corner; $f''(TTT,256n + \ ) = 137n + \ \ $.}
\end{center}

\noindent It is straightforward to compare the tables in Figures 3 and 4 to see that $f'(TTT, k)$ is greater than or equal to $f''(TTT, k)$ unless $k =5$, with equality only for $k =  0$, 1, 2, 3, 4, 6, 7, 9, 11, 12, 13, 14, 19, 20, 22, or 26, which proves Theorem~\ref{final TTT}. These tables are computed by working backwards from ending positions.  For completeness, we include tables of critical thresholds for the other positions in the trees of possible optimal moves in Figures 1 and 2.   Postions $4f$ and $5b'$ are equivalent to $\E$, position $4a$ is equivalent to $A^2$, positions $3e$, $5c$, $5d$, $5f$, $6a$, $6b$, $6c$, and $4b'$ are equivalent to $B^2$; see Example~\ref{simple tables} for tables of critical thresholds for these positions.  The critical thresholds for the remaining positions are as follows.

\vspace{10 pt}

\noindent Positions $5a$, $5b$, $5e$, $4d$, $4g$, and $5a'$ are all equivlent to $\A \wedge B^2$.

\vskip 10 pt \begin{tabular}{|r||*{8}{l|}}\hline 
$k=8n+$&0&1&2&3&4&5&6&7\\ \hline \hline
$f(\A \wedge B^2, k) = 3n + $&0*&0*&1&1*&2&2&2*&3\\ \hline 
\end{tabular}

\vspace{10 pt}

\noindent Positions $4b$, $4e$, $3c$ and $5a'$ are equivalent to $(\A \wedge B^2) \wedge \B$.  From each such position $v$, we have $f( TTT_v, 16n+\ ) = 11n + \ :$

\vskip 10 pt \begin{tabular}{|r||*{16}{l|}}\hline 
$16n+$&0&1&2&3&4&5&6&7\\ \hline \hline
0+&1&1*&2&3&3*&4&5&5*\\ \hline
8+&6*&7&7*&8*&9&9*&10*&11\\ \hline 
\end{tabular}

\vskip 10pt

\noindent Position $4c$ is equivalent to $(\A \wedge B^2) \wedge B^2$, and $f(TTT_{4c}, 16n + \ ) = 9n + \ :$

\vskip 10pt \begin{tabular}{|r||*{8}{l|}}\hline 
$16n+$&0&1&2&3&4&5&6&7\\ \hline \hline
0+&1&1&1*&2*&3&3&4&4*\\ \hline
8+&5*&5*&6&7&7*&7*&8*&9\\ \hline 
\end{tabular}

\vskip 10pt

\noindent Position $4h$ is equivalent to $B^3$.

\vskip 10pt \begin{tabular}{|r||*{8}{l|}}\hline 
$k=8n+$&0&1&2&3&4&5&6&7\\ \hline \hline
$f(B^3,k) = 7n+$&1&2&3&3*&4*&5*&6*&7\\ \hline 
\end{tabular}

\vskip 10pt

\noindent From position $3a$, we have $f(TTT_{3a}, 32n+ \ ) = 15n + \ :$

\vskip 10pt \begin{tabular}{|r||*{12}{l|}}\hline 
$32n+$&+0&+1&+2&+3&+4&+5&+6&+7&+8&+9&+10&+11\\ \hline \hline
0+&0*&1&1*&2&2&3&3*&3*&4&5&5&5*\\ \hline
12+&6&6*&7&7*&8&8*&9&9*&9*&10*&11&11\\ \hline
24+&11*&12*&12*&13&13*&14&14*&15&&&&\\ \hline 
\end{tabular}
\normalsize

\vskip 10pt

\noindent From position $3b$, we have $f(TTT_{3b}, 32n + \ ) = 9n + \ $:

\vskip 10pt \begin{tabular}{|r||*{12}{l|}}\hline 
$32n+$&+0&+1&+2&+3&+4&+5&+6&+7&+8&+9&+10&+11\\ \hline  \hline
$0+$&0*&0*&0*&1&1*&1*&2&2&2*&2*&3&3*\\ \hline
12+&3*&3*&4&4*&5&5&5&5*&6&6&6*&6*\\ \hline
24+&7&7&7*&8&8&8&8*&9&&&&\\ \hline 
\end{tabular}
\normalsize

\vskip 10pt

\noindent From position $3d$, we have $f(TTT_{3d}, 32n + \ ) = 11n + \ :$

\vskip 10pt \begin{tabular}{|r||*{12}{l|}}\hline 
$32n+$&+0&+1&+2&+3&+4&+5&+6&+7&+8&+9&+10&+11\\ \hline \hline
0+&0*&0*&1&1*&1*&2&2*&2*&3&3*&3*&4\\ \hline
12+&4*&4*&5&5*&6&6&6*&7&7&7*&8&8\\ \hline
24+&8*&9&9&9*&10&10&10*&11&&&&\\ \hline 
\end{tabular}

\vskip 10pt

\noindent Position $3f$ is equivalent to $(\A \wedge B^2) \wedge B^3$.

\vskip 10pt  \begin{tabular}{|r||*{8}{l|}}\hline 
$k=8n+$&0&1&2&3&4&5&6&7\\ \hline \hline
$f((\A \wedge B^2) \wedge B^3, k) = 5n +$ &1&1*&2&2*&3&3*&4*&5\\ \hline 
\end{tabular}

\vskip 10pt

\noindent From position $2a$, we have $f(TTT_{2a}, 64n + \ ) = 15n + \ $: \nopagebreak
\vskip 10pt \begin{tabular}{|r||*{12}{l|}}\hline 
$64n+$&+0&+1&+2&+3&+4&+5&+6&+7&+8&+9&+10&+11\\ \hline \hline
0+&0&0*&0*&1&1&1*&1*&1*&2&2*&2*&2*\\ \hline
12+&3&3&3*&3*&4&4&4*&4*&4*&5&5*&5*\\ \hline
24+&5*&6&6&6*&6*&7&7&7*&7*&8&8&8*\\ \hline
36+&8*&9&9&9&9*&10&10&10&10*&10*&11&11\\ \hline 
48+&11*&11*&12&12&12&12*&13&13&13&13*&13*&14 \\ \hline
60+&14&14*&14*&15&&&&&&&&\\ \hline 
\end{tabular}

\vskip 10pt

\noindent From position $2b$, we have $f(TTT_{2b}, 64n+ \ ) = 31n + \ $:
\vskip 10pt \begin{tabular}{|r||*{12}{l|}}\hline 
$64n+$&+0&+1&+2&+3&+4&+5&+6&+7&+8&+9&+10&+11\\ \hline \hline
0+&1&1&1*&2&2*&3&3*&3*&4*&5&5&6\\ \hline
12+&6*&6*&7&8&8*&8*&9&10&10&10*&11*&11*\\ \hline
24+&12&12*&13&13*&14&14&15&15*&16*&16*&17&17*\\ \hline
36+&18&18*&19&19&20&20*&20*&21*&22&22&22*&23*\\ \hline 
48+&24&24&24*&25*&25*&26&27&27&27*&28&28*&29\\ \hline
60+&29*&29*&30*&31&&&&&&&&\\ \hline 
\end{tabular}
\normalsize

\vskip 10pt

\noindent From position $2c$, we have $f(TTT_{2c}, 64n + \ ) = 35n + \ $:
\vskip 10pt \begin{tabular}{|r||*{12}{l|}}\hline 
$64n+$&+0&+1&+2&+3&+4&+5&+6&+7&+8&+9&+10&+11\\ \hline \hline
0+&1&1&1*&2*&3&3&4&4*&5&5*&6&6*\\ \hline
12+&7*&7*&8&9&9*&9*&10*&11&11*&12&12*&13\\ \hline
24+&14&14&14*&15*&16&16&17&17*&18*&18*&19&20\\ \hline
36+&20*&20*&21*&22&22*&23&23*&24&25&25&25*&26*\\ \hline
48+&27&27&28&28*&29&29*&30&30*&31*&31*&32&33\\ \hline
60+&33*&33*&34*&35&&&&&&&&\\ \hline 
\end{tabular}
\normalsize

\vskip 10pt

\noindent From position $2d$, we have $f(TTT_{2d}, 16n + \ ) = 13n + \ $:
\vskip 10pt \begin{tabular}{|r||*{8}{r|}}\hline 
$16n+$&0&1&2&3&4&5&6&7\\ \hline \hline
0+&1&2&2*&3*&4&5&6&6*\\ \hline
8+&7*&8*&9&10&10*&11*&12*&13\\ \hline 
\end{tabular}

\vskip 10pt

\noindent From position $1a$, we have $f(TTT_{1a}, 64n + \ ) = 23n + \ $: \nopagebreak
\vskip 10pt \begin{tabular}{|r||*{12}{l|}}\hline 
$64n+$&+0&+1&+2&+3&+4&+5&+6&+7&+8&+9&+10&+11\\ \hline \hline
0+&0*&1&1&1*&1*&2*&2*&2*&3&4&4&4*\\ \hline
12+&4*&4*&5*&6&6&6&7&7*&7*&7*&8*&8*\\ \hline
24+&9&9&9*&10&10*&10*&11&11*&12&12&12*&13\\ \hline
36+&13*&13*&14&14&15&15&15&15*&16*&16*&16*&17\\ \hline 
48+&18&18&18&18*&18*&19*&20&20&20&21&21&21*\\ \hline
60+&21*&22&22*&23&&&&&&&&\\ \hline 
\end{tabular}

\vskip 10pt

\noindent From position $1b$, we have $f(TTT_{1b}, 128n + \ ) = 87n + \ $:
\vskip 10pt \begin{tabular}{|r||*{12}{l|}}\hline 
$128n+$&+0&+1&+2&+3&+4&+5&+6&+7&+8&+9&+10&+11\\ \hline \hline
0+&1&1*&2&3&3*&4&5&5*&6&7&7*&8*\\ \hline
12+&9&9*&10&11&12&12*&13&13*&14*&15&16&16\\ \hline
24+&17&17*&18*&19*&19*&20&21&22&23&23*&23*&24*\\ \hline
36+&25*&26&27&27&28&28*&29*&30&30*&31&32&33\\ \hline 
48+&33*&34&34*&35*&36&37&37*&38&39&39*&40&41\\ \hline
60+&41*&42&43&43*&44*&45&45*&46*&47&47*&48*&49\\ \hline
72+&49*&50*&51&52&52*&53&53*&54*&55*&56&56*&57\\ \hline
84+&58&58*&59*&59*&60*&61&62&63&63&63*&64*&65*\\ \hline 
96+&66*&67&67&68&69&69*&70*&70*&71*&72&73&73*\\ \hline
108+&74&74*&75*&76*&77&77*&78&79&79*&80*&81&81*\\ \hline
120+&82*&83&83*&84*&85&85*&86*&87&&&&\\ \hline 
\end{tabular}

\vskip 10pt

\noindent From position $4a'$, we have $f(TTT_{4a'}, 16n + \ ) = 3n+ \ $:
\vskip 10pt \begin{tabular}{|r||*{8}{r|}}\hline 
$16n+$&0&1&2&3&4&5&6&7\\ \hline \hline
0+&0&0&0*&0*&1&1&1&1*\\ \hline
8+&1*&1*&2&2&2*&2*&2*&3\\ \hline 
\end{tabular}

\vskip 10pt

\noindent From position $3a'$, we have $f(TT_{3a'}, 32n + \ ) = 15n + \ $:
\vskip 10pt \begin{tabular}{|r||*{12}{l|}}\hline 
$32n+$&+0&+1&+2&+3&+4&+5&+6&+7&+8&+9&+10&+11\\ \hline \hline
0+&0*&0*&1*&2&2*&2*&3&4&4*&4*&5&5*\\ \hline
12+&6*&6*&7&7*&8&8&9&9*&10&10&10*&11*\\ \hline
24+&12&12&12*&13&14&14&14*&15&&&&\\ \hline 
\end{tabular}

\vskip 10pt

\noindent From position $2a'$, we have $f(TTT_{2a'}, 64n + \ ) = 15n + \ $: \nopagebreak
\vskip 10pt \begin{tabular}{|r||*{12}{l|}}\hline 
$64n+$&+0&+1&+2&+3&+4&+5&+6&+7&+8&+9&+10&+11\\ \hline \hline
0+&0&0&0*&1&1&1&1*&2&2&2&2*&2*\\ \hline
12+&3&3&3*&3*&4&4&4*&4*&5&5&5&5*\\ \hline
24+&6&6&6&6*&7&7&7&7*&7*&7*&8&8*\\ \hline
36+&8*&8*&9&9*&9*&9*&10&10&10*&10*&11&11\\ \hline 
48+&11*&11*&12&12&12*&12*&12*&13&13*&13*&13*&14\\ \hline
60+&14*&14*&14*&15&&&&&&&&\\ \hline 
\end{tabular}

\vskip 10pt

\noindent From position $1a'$, we have $f(TTT_{1a'}, 64n + \ ) = 25n + \ $: \nopagebreak
\vskip 10pt \begin{tabular}{|r||*{12}{l|}}\hline 
$64n+$&+0&+1&+2&+3&+4&+5&+6&+7&+8&+9&+10&+11\\ \hline \hline
0+&0*&0*&1&1*&2&2&3&3&3*&3*&4*&4*\\ \hline
12+&5&5&6&6*&6*&6*&7*&8&8&8*&8*&9*\\ \hline
24+&10&10&10&11&11*&11*&12&12*&13&13&13*&14*\\ \hline
36+&14*&14*&15&16&16&16*&16*&17&18&18&18&18*\\ \hline 
48+&19*&19*&20&20&21&21&21*&21*&22*&22*&23&23*\\ \hline
60+&24&24&24*&25&&&&&&&&\\ \hline 
\end{tabular}

\vskip 10pt

\noindent This concludes our analysis of bidding Tic-Tac-Toe. \end{proof}

\section{Appendix: other tie-breaking methods}

Throughout this paper we have used the tie-breaking method introduced in Section \ref{tie-breaking section}, in which the player who holds the tie-breaking chip can either give the chip to his opponent to win a tie, or keep the tie-breaking chip and lose the tie.  We chose this method because it seemed simple and natural, with the tie-breaking advantage passing back and forth between the players just like the ordinary bidding chips, and because it has the following convenient properties.
\begin{enumerate}
\item It is always advantageous to have the tie-breaking chip (Lemma \ref{* is an advantage}).
\item The tie-breaking chip is worth less than an ordinary chip (Lemma \ref{1>*}).
\end{enumerate}
Other natural tie-breaking methods are possible, and here we briefly consider a few alternatives.

\vspace{10 pt}

\noindent \textbf{Loser's Ball (the $\epsilon$-chip)}.  We could make a rule that the player who loses one bid wins any tie on the next bid.  One way to think about this tie-breaking method is by introducing an \emph{$\epsilon$-chip} whose value is strictly between zero and one.  The player who holds the $\epsilon$-chip is required to bid it, so the bids are never tied.  Whichever player loses the bid takes all of the chips that were bid, and hence has the $\epsilon$-chip for the next round. This method has the mildly unpleasant feature that holding the $\epsilon$-chip is usually but not always an advantage.  For instance, in the game where the second player to move wins, if both players start with an equal number of chips then the player who starts without the $\epsilon$-chip wins.  In particular, the possible chip counts in $\N \cup \N + \epsilon$ with respect to this tie-breaking method are not totally ordered.

\vspace{10 pt}

\noindent \textbf{Make-it Take-it (the $-\epsilon$-chip).} A better idea than Loser's Ball is to make a rule that the player who wins one bid also wins any ties on the next bid.  This method is satisfying, since it penalizes players for losing bids and often leads to taunting.  One way to think about this tie-breaking method is by introducing a \emph{$-\epsilon$ chip} whose value is strictly between zero and minus one.  The player who holds the $-\epsilon$-chip is required to bid it, so the bids are never tied.  Whichever player loses the bid takes the $-\epsilon$-chip for the next round, and hence loses any tie.  Arguments similar to those given in Section~\ref{general theory} show that it is always an advantage \emph{not} to have the $-\epsilon$-chip and that it is always a good idea to accept the $-\epsilon$-chip from your opponent together with an ordinary chip.  Therefore, the possible chip counts in $\N \cup \N - \epsilon$ are totally ordered, and results analogous to those in Section~\ref{general theory} go through without major changes, except that the analogue of the recursion in Theorem~\ref{discrete recursion} is given by
\[
f(G_v,k) = \left \lfloor \frac{ | f_A(G_v, k) | +  | f_B(G_{v},k) | }{ 2 } \right \rfloor + \delta,
\]
where 
\[
\delta = \left \{ \begin{array}{ll} 0 & \mbox{ if } |f_A(G_v,k)| + |f_B(G_{v},k)| \mbox{ is even, and } f_B(G_v, k) \in \N. \\
			 		    -\epsilon & \mbox { if } |f_A(G_v,k)| + |f_B(G_{v},k)| \mbox{ is even, and } f_B(G_v, k) \in \N - \epsilon. \\
					  1-\epsilon & \mbox{ if } |f_A(G_v,k)| + |f_B(G_{v},k)| \mbox{ is odd, and } f_A(G_v, k) \in \N. \\
				 		      0  & \mbox{ if } |f_A(G_v,k)| + |f_B(G_{v},k)| \mbox{ is odd, and } f_A(G_v, k) \in \N - \epsilon. \\
								    \end{array} \right.
\]
The main disadvantage to the Make-it Take-it tie-breaking method seems to be that this description involves a lot of minus signs.

\vspace{10 pt}

\noindent \textbf{Ladies First.}  Suppose Alice wins all ties.  The longer Alice and Bob play, the more the effects of Alice's advantage accumulate, so we should expect that Bob will have trouble if the game goes too long.

Let $G$ be the game that Alice wins if she wins the first move, or both of the next two moves, or all of the next three moves, or the $i$-th move for all $n(n-1)/ 2 < i \leq n(n+1) / 2$, for any $n$.  Then Bob cannot win, but we can hope to prolong the game into an infinite draw.  If $G$ is truncated after $n(n+1)/2$ moves, then the critical threshold is
\[
R(G[n(n+1)/2]) = \prod_{j = 1}^n \big(1 - \frac{1}{2^j}\big).
\]
In particular, the critical threshold for Richman play, with real-valued bidding is given by the infinite product $R(G) = \prod_{j > 0} (1 - 1/2^j)$, which converges to a positive number.  However, for discrete bidding with Ladies First tie-breaking, Alice wins no matter what the chip count.  Alice can just bid zero indefinitely.  Bob will have to give her a chip sometime between the moves numbered $n(n-1)/2$ and $n(n+1)/2$ for every $n$, until he runs out of chips.  When Bob runs out of chips, Alice keeps bidding zero and winning all ties, so she eventually wins.  In particular, Theorem~\ref{below the threshold} does not extend to locally finite games with Ladies First tie-breaking.

\bibliography{math}
\bibliographystyle{amsalpha}

\end{document}